\documentclass[a4paper,15pt]{article}
\usepackage[american]{babel}
\usepackage[utf8]{inputenc}
\usepackage{cite}
\usepackage{geometry}
\usepackage{amsmath, amsthm, latexsym, amssymb}
\usepackage{amsfonts}
\usepackage{booktabs}  
\usepackage{graphicx} 
\usepackage{listings}
\usepackage{hyperref}
\usepackage{anysize}
\usepackage{caption}
\usepackage{subcaption}
\usepackage{placeins}
\usepackage{mathtools}
\usepackage{todonotes}
\usepackage{tcolorbox}
\usepackage{rotating,tabularx,siunitx}
\usepackage[title]{appendix}

\def\R{\mathbb{R}}

\def\elec{\mathbf{E}}
\def\magn{\mathbf{B}}
\def\D{\mathbf{D}}
\def\F{\mathbf{F}}
\def\U{\mathbf{U}}
\def\W{\mathbf{W}}
\def\curre{\mathbf{J}}
\def\vecx{\mathbf{x}}
\def\vecv{\mathbf{v}}
\def\unitv{\mathbf{e}}
\def\normalx{\mathbf{n}_{x}}
\def\normalv{\mathbf{n}_{v}}
\def\rot{\nabla_{\vecx}\times}
\def\Th{\mathcal{T}_h}
\def\Thx{\mathcal{T}_h^x}
\def\Thv{\mathcal{T}_h^v}
\def\edgesx{\mathcal{E}_x}
\def\edgesv{\mathcal{E}_v}
\def\edges{\mathcal{E}}
\newtheorem{thm}{Theorem}
\newtheorem{lem}[thm]{Lemma}

\newcommand{\norma}[1]{\|{#1}\|}
\newcommand{\einspace}[1]{\zeta_h^{#1}}

\usepackage[wby]{callouts}

\title{Superconvergence and accuracy enhancement of discontinuous Galerkin solutions for Vlasov-Maxwell equations}
\author{
Andr\'{e}s Galindo-Olarte
\thanks{Department of Mathematics, Michigan State University,
East Lansing, MI 48824 U.S.A.
 {\tt galindoo@msu.edu}}
\and
Juntao Huang
\thanks{Department of Mathematics and Statistics, Texas Tech University, Lubbock, TX 70409 U.S.A.
{\tt juntao.huang@ttu.edu}}
\and
Jennifer Ryan
\thanks{Department of Mathematics, KTH Royal Institute of Technology, 100 44 Stockholm, Sweden.{\tt jryan@kth.se}. Research supported by the Air Force Office of Scientific Research (AFOSR), Computational Mathematics program (program manager Fariba Fahroo), under grant number FA-9550-20-1-0166.}
\and
Yingda Cheng
\thanks{Department of Mathematics, Department of  Computational Mathematics, Science and Engineering, Michigan State University,
East Lansing, MI 48824 U.S.A.
 {\tt ycheng@msu.edu}. Research is supported by NSF grant DMS-2011838.}
}

\date{}

\begin{document}
	\maketitle
	
\begin{abstract}
 This paper considers the discontinuous Galerkin (DG) methods for solving the  Vlasov-Maxwell (VM) system, a fundamental model for  collisionless magnetized plasma. The DG methods provide accurate numerical description with conservation and stability properties. However,  to resolve the high dimensional probability distribution function, the computational cost is the main bottleneck even for modern-day supercomputers. This work studies the applicability of a post-processing technique to the DG solution to enhance its accuracy and resolution for the VM system. In particular, we prove the superconvergence of   order $(2k+\frac{1}{2})$ in the negative order norm for the probability distribution function and the electromagnetic fields when piecewise polynomial degree $k$ is used. Numerical tests including Landau damping, two-stream instability and streaming Weibel instabilities are considered showing the performance of the post-processor.
\end{abstract}

	\section{Introduction}
	 	
	In this paper, we consider numerical solutions of the  Vlasov-Maxwell (VM) system, a fundamental model for  collisionless magnetized plasma. The dimensionless form of the equations that describes the evolution of a single species of non-relativistic electrons under the self-consistent electromagnetic field while the ions are treated as uniform fixed background is given by
	 	\begin{subequations}\label{eq:vm_system}
		\begin{gather}
			\partial_t f+\mathbf{v}\cdot\nabla_{\mathbf{x}} f+(\elec+\mathbf{v}\times\magn)\cdot\nabla_{\mathbf{v}} f=0,\label{eq:vm_main_eq}\\
			\frac{\partial\elec}{\partial t}=\nabla_\mathbf{x}\times\magn-\curre,\quad\frac{\partial\magn}{\partial t}=-\nabla_\mathbf{x}\times\elec,\label{eq:temp_elec_magn}\\
			\nabla_{\mathbf{x}}\cdot\elec=\rho-\rho_i,\quad\nabla_{\mathbf{x}}\cdot\magn=0, 
		\end{gather}
	\end{subequations}
	with 
	\begin{equation*}
		\rho(\mathbf{x},t)=\int_{\Omega_{v}}f(\mathbf{x},\mathbf{v},t)\,d\mathbf{v},\quad\curre(\mathbf{x},t)=\int_{\Omega_v}f(\mathbf{x},\mathbf{v},t)\mathbf{v}\,d\mathbf{v},
	\end{equation*}
	where the equations are defined on $\Omega=\Omega_x\times\Omega_v$, $\vecx\in\Omega_x$ denotes the position in physical space, and $\vecv\in\Omega_v$ in velocity space. Here $f(\vecx,\vecv,t)\geq 0$ is the distribution function of electrons at position $\vecx$ with velocity $\vecv$ at time $t$, $\elec(\vecx,t)$ is the electric field, $\magn(\vecx,t)$ in the magnetic field, $\rho(\vecx,t)$ is the electron charge density, and $\curre(\vecx,t)$ is the current density. The charge density of background ions is denoted by $\rho_i$, which is chosen to satisfy total charge neutrality, \mbox{$\int_{\Omega_x}(\rho(\vecx,t)-\rho_i)\,d\vecx=0$}. Periodic boundary conditions in $\Omega_x$ and compact support in $\Omega_v$ are assumed. 
The VM system has wide applications in plasma physics for describing space and laboratory plasmas, with application to fusion devices, high-power microwave generators, and large scale particle accelerators. 

Much work has been carried out in the literature aiming at accurate deterministic description of the probability density function for nonlinear behavior of charged particles in plasma. 
 Califano \emph{et al.}   used a semi-Lagrangian approach to compute the  streaming Weibel  instability \cite{califano1998ksw}, current filamentation instability \cite{mangeney2002nsi}, magnetic vortices \cite{califano1965ikp}, magnetic reconnection \cite{califano2001ffm}. Also, various methods have  been  proposed for the relativistic VM system \cite{Sircombe20094773, Besse20087889, Suzuki20101643, Huot2003512}. This work concerns the discontinuous Galerkin (DG) method  for solving  the VM system.
The DG method is a class of finite element method that uses discontinuous polynomial spaces, and they  have
 desirable properties for convection-dominated problems \cite{Cockburn_2001_RK_DG}.
In particular, DG methods have  been used to simulate the Vlasov-Poisson system in plasmas \cite{Heath, Heath_thesis, Cheng_Gamba_Morrison} and for a  gravitational infinite homogeneous stellar system \cite{Cheng_jeans}.  They have been also used to solve VM system \cite{cheng2014discontinuous,cheng2014energy} and the relativistic VM system \cite{yang2017discontinuous}. The DG methods have nice properties such as stability, charge and energy conservation and high order accuracy, which are highly desirable for long time simulations.

The main computational challenge for any grid based solver for the VM system  is  the high-dimensionality of the Vlasov equation. This makes the computation extremely expensive even on modern-day exa-scale supercomputers.  Post-processing techniques, which can greatly enhance the resolution of the numerical solution at any given time, are therefore desirable because it is only applied once at the end of the simulation with negligible computational cost. Post-processing for finite element methods is a mature technology. The post-processing technique presented here takes advantage of the information contained in the negative-order norm and was originally developed by Bramble and Schatz \cite{bramble1977higher} in the context of continuous finite element methods for elliptic problems. It consists of a convolution of the finite element solution with a local averaging operator. We can then establish the convergence in the negative order norm which is higher than that one obtained in the usual $L^2$-norm.  In  \cite{cockburn2003enhanced},  Cockburn, Luskin, Shu and S{\"u}li applied this technique to the DG methods for solving linear hyperbolic equations. This technique was further extended to the DG methods for solving nonlinear conservational laws \cite{ji2013negative,meng2017discontinuous} and nonlinear symmetric systems of hyperbolic conservation laws \cite{meng2018divided}.  This method is currently part of a filtering family known as a Smoothness-Increasing Accuracy-Conserving (SIAC) filters \cite{steffan2008investigation}. This paper will demonstrate the performance of post-processing by the SIAC filter for DG solutions to the VM system. In particular, we consider benchmark numerical tests for Vlasov-Amp\'{e}re (VA) and VM systems, and study the numerical error for short and long time simulations with varying polynomial order.
 
 
In order to validate the enhanced accuracy of the post-processed solution, an important step is to establish the superconvergence of the negative order norm of the error and its divided differences.  In \cite{cockburn2003enhanced},  Cockburn, Luskin, Shu and S{\"u}li established a framework to prove negative-order estimates for the DG solutions to linear conservational laws of order $2k+1$ using polynomials of degree $k.$ After this, there have been important extensions. $L^2$ and $L^{\infty}$ superconvergence estimates were established for DG solutions for linear constant coefficient hyperbolic systems    with the position-dependent SIAC filter \cite{ji2014superconvergent}. Ji, Meng \emph{et al} \cite{ji2013negative, meng2017discontinuous,meng2018divided} proved superconvergence for non-linear conservation laws  and  nonlinear symmetric hyperbolic systems  of the DG solutions of order at least $(\frac{3}{2}k+1)$.  It is highly nontrivial to establish superconvergence for nonlinear problems because a suitable dual problem has to be identified, and additionally the divided difference of the solution does not satisfy the PDE, which makes the proof highly technical \cite{meng2017discontinuous,meng2018divided}. In this work, we aim to prove negative-order estimates   of DG solutions to the VM system. Since the VM system is nonlinear, it is nontrivial to extend the proof in  \cite{cockburn2003enhanced}. We identify a proper dual problem, which aids the estimates of the consistency term. In the end, we proved superconvergence of order $(2k+\frac{1}{2})$ in the negative norm for the probability distribution function and the electromagnetic fields.
     
	
The paper is organized as follows. In Section \ref{sec:dg_num_sche}, we introduce the DG method for the VM system as well as relevant notations that will be required for the negative order estimates. In Section \ref{sec:siac_filter} we introduce SIAC filtering. In Section \ref{sec:sc_dg_method} we prove the negative-order norm estimates of the DG solutions to the VM system. The superconvergence results are confirmed numerically in Section \ref{sec:num_experiments}.   We conclude the paper with remarks and future work in Section \ref{sec:conclude}. 

	\section{Discontinuous Galerkin Numerical Scheme}\label{sec:dg_num_sche}
	\subsection{Notations, Definitions and Projections}
		We begin by introducing the necessary notation used in the paper. 
		Without loss of generality, we assume the spatial and velocity domain to be $\Omega_x=[-L_x,L_x]^{d_x}$ and $\Omega_{v}=[-L_v,L_v]^{d_v}$, where  $L_v$ is chosen large enough so that $f=0$ at  $\partial \Omega_v.$ 
	Through out the paper, standard notations will be used for the Sobolev spaces. Given a bounded domain $D\in\R^{\star}$ (with $\star=d_x$,$d_v$, or $d_x+d_v$) and any nonnegative integer $m$, $H^m(D)$ denotes the $L^2$-Sobolev space of order $m$ with the standard Sobolev norm $\norma{\cdot}_{m,D}$, $W^{m,\infty}$ denotes  the $L^{\infty}$-Sobolev space of order $m$ with the standard Sobolev norm $\norma{\cdot}_{m,\infty, D}$ and the semi-norm $|\cdot|_{m,\infty,D}$. When $m=0$, we also use $H^{0}(D)=L^2(D)$ and $W^{0,\infty}(D)=L^{\infty}(D)$.
	
		Let $\Thx=\{K_x\}$ and $\Thv=\{K_v\}$ be partitions of $\Omega_x$ and $\Omega_v$, respectively, with $K_x$ and $K_v$ being  Cartesian elements or simplices; then $\Th=\{K:K=K_x\times K_v,\,\forall K_x\in\Thx,\,\forall K_v\in\Thv\}$ defines a partition of $\Omega$. Let $\edgesx$ be the set of the edges of $\Thx$ and $\edgesv$ the set of the edges of $\Thv$; then the edges of $\Th$ will be $\edges=\{K_x\times e_v:\forall K_x \in\Thx,\,\forall e_v\in\edgesv\}\cup\{e_x\times K_v:\,\forall e_x\in\edgesx,\forall K_x \in\Thx\}.$  
Furthermore, $\edgesv=\edgesv^i\cup\edgesv^b$ with $\edgesv^i$ and $\edgesv^b$ being the set of interior and boundary edges of $\Thv$ respectively. In addition, we denote the mesh size of $\Th$ as $h=\max(h_x,h_v)=\max_{K\in\Th}h_K$, where $h_x=\max_{K_x\in\Thx}h_{K_x}$ with $h_{K_x}=\mathrm{diam}(K_x)$, $h_v=\max_{K_v\in\Thv}h_{K_v}$ with $h_{K_v}=\mathrm{diam}(K_v)$, and $h_K=\max(h_{K_x},h_{K_v})$ for $K=K_x\times K_v$.  When the mesh is refined, we assume both  $\frac{h_x}{h_{x,\min}}$ and $\frac{h_v}{h_{v,\min}}$ are uniformly bounded from above by a positive constant $\sigma_0$. Here $h_{x,\min}=\min_{K_x}h_{K_x\in\Thx}$ and $h_{v,\min}=\min_{K_v\in\Thv}h_{K_v}$. It is further assumed that $\left\{\mathcal{T}_h^{\star}\right\}_h$ is shape-regular with $\star=x$ or $v$. That is, if $\rho_{K_{\star}}$ denotes the diameter of the largest sphere included in $K_{\star}$,  there is 
	\begin{equation*}
		\frac{h_{K_{\star}}}{\rho_{K_{\star}}}\leq \sigma_{\star},\quad\forall K_{\star}\in\mathcal{T}_h^{\star} 
	\end{equation*}
for a positive constant $\sigma_{\star}$ independent of $h_{\star}$. 
	Furthermore the inner products are defined as 
	\begin{gather}
		(g,h)_{\Omega}=\int_{\Omega}gh\,dx\,dv=\sum_{K\in\Th}\int_{K}gh\,dx\,dv,\\
		(\U,\mathbf{W})_{\Omega_x}=\int_{\Omega_x}\U\cdot\mathbf{W}\,dx=\sum_{K_x\in\Thx}\int_{K_x}\U\cdot\mathbf{W}\,dx. 
	\end{gather}
	
	Now for $g\in L^2(\Omega)$, $\U,\W\in (L^2(\Omega_x))^{d_x}$, we define the $L^2$-norm of $(g,\U,\W)$ as
	\begin{equation}
	\norma{(g,\U,\W)}_{0,\Omega}=\sqrt{\norma{g}_{0,\Omega}^2+\norma{\U}_{0,\Omega_x}^2+\norma{\W}_{0,\Omega_x}^2}
	\end{equation}
	This will be helpful in the error analysis of the negative-order norm. The negative order norm is defined as: given $l>0$ and domain $\Omega$, 
    \begin{equation*}
    	\norma{(g,\U,\mathbf{W})}_{-l,\Omega}=\sup_{\phi\in C_0^{\infty}(\Omega),\mathcal{U},\mathcal{W}\in[C_0^{\infty}(\Omega_x)]^{d_x}}\frac{(g,\phi)_{\Omega}+(\U,\mathcal{U})_{\Omega_x}+(\mathbf{W},\mathcal{W})_{\Omega_x}}{\sqrt{\norma{\phi}_{l,\Omega}^2+\norma{\mathcal{U}}_{l,\Omega_x}^2+\norma{\mathcal{W}}_{l,\Omega_x}^2}}
    \end{equation*}

    Next we define the discrete spaces 
    \begin{align}
    	\mathcal{G}_h^k&=\left\{g\in L^2(\Omega):\left.g\right|_{K=K_x\times K_v}\in P^k(K_x\times K_v),\forall K_x\in\Thx,\forall K_x\in\Thx,\forall K_v\in\Thv,\right\}\label{eq:f_space}\\
    	               &=\left\{g\in L^2(\Omega):\left.g\right|_{K}\in P^k(K),\forall K\in\Th\right\},\nonumber\\
    	\mathcal{U}_h^r&=\left\{\U\in\left[L^2(\Omega_x)\right]^{d_x}:\left.\U\right|_{K_x}\in\left[P^r(K_x)\right]^{d_x},\forall K_x\in\Thx\right\}\label{eq:fields_space},
    \end{align}
	where $P^r(D)$ denotes the set of polynomials of total degree at most $r$ on $D$, and $k$ and $r$ are nonnegative integers. 
	
	For piecewise functions defined with respect to $\Thx$ or $\Thv$, we further introduce the jumps and averages as follows. For any edge $e=\left\{K_x^+\cap K_x^{-}\right\}\in\edgesx$, with $\normalx^{\pm}$ as the outward unit normal to $\partial K_{x}^{\pm}$, $g^{\pm}=\left.g\right|_{K_x^{\pm}}$ and $\U^{\pm}=\left.\U\right|_{K_x^{\pm}}$, the jump across $e$ are defined as 
	\begin{equation*}
		[g]_x=g^{+}\normalx^+ +g^{-}\normalx^{-},\quad[\U]_x=\U^{+}\cdot\normalx^+ +\U^{-}\cdot\normalx^{-},\quad[\U]_{\tau}=\U^{+}\times\normalx^+ +\U^{-}\times\normalx^{-}
	\end{equation*}
and the averages are 
 	\begin{equation*}
 		\{g\}_x=\frac{1}{2}(g^++g^-),\quad\{\U\}_x=\frac{1}{2}(\U^+ + \U^-).
 	\end{equation*}
	By replacing the subscript $x$ with $v$, one can define $[g]_v,\,[\U]_v,\,\{g\}_v$, and $\{\U\}_v$ for an interior edge of $\Thv$ in $\edgesv^i$. For a boundary edge $e\in\edgesv^b$ with $\normalv$ being the outward unit normal we use 
	\begin{equation}
		[g]_v=g\normalv,\quad\{g\}_v=\frac{1}{2}g,\quad\{\U\}_v=\frac{1}{2}\U.
		\label{eq:bound_aver_jump}
	\end{equation}
	This is consistent with the fact that the exact solution $f$ is compactly supported in $\vecv$. 
	
	For convenience, we introduce some shorthand notations, $\int_{\Omega_{\star}}=\int_{\Th^{\star}}=\sum_{K_{\star}\in\Th^{\star}}\int_{K_{\star}}$, $\int_{\Omega}=\int_{\Th}=\sum_{K\in\Th}\int_{K}$, $\int_{\edges_{\star}}=\sum_{e\in\edges_{\star}}\int_{e}$, where again $\star$ is $x$ or $v$. In addition, $\norma{g}_{0,\edges}=(\norma{g}_{0,\edgesx\times\Thv}^2+\norma{g}_{0,\Thx\times\edgesv}^2)^{1/2}$ with $\norma{g}_{0,\edgesx\times\Thv}=\left(\int_{\edgesx}\int_{\Thv}g^2\,d\vecv\,ds_{\vecx}\right)^{1/2}$, $\norma{g}_{0,\Thx\times\edgesv}=\left(\int_{\Thx}\int_{\edgesv}g^2\,ds_{\vecv}\,d\vecx\right)^{1/2}$.
	We will make use of the following equality, which can be easily verified using the definition of averages and jumps.
	\begin{gather}
		\frac{1}{2}[g^2]_{\star}={g}_{\star}[g]_{\star},\text{with }\star=x \text{ or }v.
	\end{gather}

\subsection{The DG method for the Vlasov-Maxwell system}\label{sec:dg_method}
	Now we review the DG method for the VM system proposed in \cite{cheng2014discontinuous}.
The scheme seeks a numerical solution $f_h\in\mathcal{G}^k_h$ and $(\elec_h,\magn_h)\in\mathcal{U}^k_{h}\times \mathcal{U}^k_{h}$ such that for any $g\in\mathcal{G}_h^k$, $\mathbf{U},\mathbf{W}\in\mathcal{U}^k_{h}$, 
	\begin{subequations}\label{eq:dg_system}
		\begin{align}		
			\int_K \partial_t f_hg\,d\vecx d\vecv&-\int_Kf_h\vecv\cdot\nabla_x g\,d\vecx d\vecv-\int_K f_h(\elec_h+\vecv\times\magn_h)\cdot\nabla_{\mathbf{v}}g\,d\vecx d\vecv\nonumber\\
			&+\int_{K_v}\int_{\partial K_x}\widehat{f_h\vecv\cdot\normalx} g\,ds_{\vecx} d\vecv+\int_{K_x}\int_{\partial K_v} \widehat{(f_h(\elec_h+\vecv\times\magn_h)\cdot\normalv)}g\,ds_{\vecv}d\vecx=0,\label{eq:dg_distribution}\\
			\int_{K_x}\partial_t\elec_h\cdot \mathbf{U}\,d\vecx&=\int_{K_x}\magn_h\cdot\rot\mathbf{U}\,d\vecx+\int_{\partial K_x}\widehat{\normalx\times\magn_h}\cdot\mathbf{U} \,ds_{\vecx}-\int_{K_x}\curre_h\cdot\mathbf{U}\,d\vecx,\label{eq:dg_electric}\\
			\int_{K_x}\partial_t\magn_h\cdot \mathbf{W}\,d\vecx&=-\int_{K_x}\elec_h\cdot\rot\mathbf{W}\,d\vecx-\int_{\partial K_x}\widehat{\normalx\times\elec_h}\cdot\mathbf{W}\,ds_{\vecx}\label{eq:dg_magnetic}
		\end{align}
	\end{subequations}
	with 
	\begin{equation}
		\curre_h(\vecx,t)=\int_{\mathcal{T}^{\vecv}_{h}}f_h(\vecx,\vecv,t)\vecv\,d\vecv.
		\label{eq:dg_current}
	\end{equation}
	Here $\normalx$ and $\normalv$ are outward unit normals of $\partial K_x$ and $\partial K_v$, respectively. All ``hat'' functions are numerical fluxes that are determined by upwinding, i.e., 
	\begin{subequations}
		\begin{align}
			\widehat{f_h\vecv\cdot\normalx}&\coloneqq\widetilde{f_h\vecv}\cdot\normalx=\left(\{f_h\vecv\}_{x}+\frac{|\vecv\cdot\normalx|}{2}[f_h]_x\right)\cdot\normalx\label{eq:flux_f_h}\\
			\widehat{f_h(\elec_h+\vecv\times\magn_h)\cdot\normalv}&\coloneqq f_h\widetilde{(\elec_h+\vecv\times \magn_h)}\cdot\normalv\nonumber\\
			&=\left(\{f_h(\elec_h+\vecv\times \magn_h)\}_{v}+\frac{|(\elec_h+\vecv\times \magn_h)\cdot\normalv|}{2}[f_h]_v\right)\cdot\normalv,\label{eq:flux_fEBh}\\
			\widehat{\normalx\times\elec_h}&\coloneqq \normalx\times\widetilde{\elec_h}=\normalx\times\left(\{\elec_h\}_x+\frac{1}{2}\left[\magn_h\right]_{\tau}\right)\label{eq:flux_Eh}\\
			\widehat{\normalx\times\magn_h}&\coloneqq \normalx\times\widetilde{\magn_h}=\normalx\times\left(\{\magn_h\}_x-\frac{1}{2}\left[\elec_h\right]_{\tau}\right)\label{eq:flux_Bh}
		\end{align}
	\end{subequations}
	where these relations define the meaning of ``tilde''. In \cite{cheng2014discontinuous}, alternating and central fluxes for the Maxwell's equation are also considered. The discussions will be similar to what will be presented in the paper for the upwind flux, and thus are omitted.
	
	Upon summing up \eqref{eq:dg_distribution} with respect to $K\in\mathcal{T}_h$ and similarly summing \eqref{eq:dg_electric} and \eqref{eq:dg_magnetic} with respect to $K_x\in\mathcal{T}_h^x$, the scheme \eqref{eq:dg_system} becomes the following: look for $f_h\in\mathcal{G}_h^k,\,\elec_h,\magn_h\in\mathcal{U}_h^k$, such that
	\begin{subequations}
		\begin{gather}
			((f_h)_t,g)_{\Omega}+a_h(f_h,\elec_h,\magn_h;g)=0\label{eq:scheme_f}\\
			((\elec_h)_t,\mathbf{U})_{\Omega_x}+((\magn_h)_t,\mathbf{W})_{\Omega_x}+b_h(\elec_h,\magn_h;\mathbf{U},\mathbf{W})=l_h(\curre_h;\mathbf{U}),\label{eq:scheme_fields}
		\end{gather}
	\end{subequations}
	for any $g\in\mathcal{G}^k_h,\,\mathbf{U},\mathbf{W}\in\mathcal{U}_{h}^k$, where
	\begin{align*}
		a_h(f_h,\elec_h,\magn_h;g)&=a_{h,1}(f_h;g)+a_{h,2}(f_h,\elec_h,\magn_h;g),\quad l_h(\curre_h;\mathbf{U})=-\int_{\mathcal{T}_h^x}\curre_h\cdot\mathbf{U}\,d\vecx,\\
		b_h(\elec_h,\magn_h;\mathbf{U},\mathbf{W})&=-\int_{\Thx}\magn_h\cdot\rot\mathbf{U}\,d\vecx-\int_{\edgesx}\widetilde{\magn_h}\cdot\left[\mathbf{U}\right]_{\tau}\,ds_x\\
		&+\int_{\Thx}\elec_h\cdot\rot\mathbf{W}\,d\vecx+\int_{\edgesx}\widetilde{\elec_h}\cdot\left[\mathbf{W}\right]_{\tau}\,ds_x,
	\end{align*}
	and 
	\begin{align*}
		a_{h,1}(f_h;g)&=-\int_{\Th}f_h\vecv\cdot\nabla_{\vecx}g\,d\vecx d\vecv+\int_{\Thv}\int_{\edgesx}\widetilde{f_h\vecv}\cdot\left[g\right]_x\,ds_x d\vecv\\
		a_{h,2}(f_h,\elec_h,\magn_h;g)&=-\int_{\Th}f_h\left(\elec_h+\vecv\times\magn_h\right)\cdot\nabla_{\vecv}g\,d\vecx d\vecv+\int_{\Thx}\int_{\edgesv}\widetilde{f_h(\elec_h+\vecv\times\magn_h)}\cdot[g]_{v}\,ds_{v}d\vecx
	\end{align*}
	
	The semi-discrete formulation \eqref{eq:dg_system} can then be solved by a numerical ODE solver, see the description in \cite{cheng2014discontinuous}. The $L^2$ and energy stability of \eqref{eq:dg_system} are established in \cite{cheng2014discontinuous}.
	The main result in \cite{cheng2014discontinuous} for the semi-discrete $L^2$ error estimates of the approximations $f_h$, $\elec_h$, $\magn_h$, is as follows.
	\begin{thm}[\cite{cheng2014discontinuous}]\label{thm:main_approx_result}
		For $k\geq 2$ when $d_x=3$ and $k\geq 1$ when $d_x=1, 2$, the semi-discrete DG method of \eqref{eq:scheme_f}-\eqref{eq:scheme_fields}, for the Vlasov-Maxwell equations with the upwind fluxes of \eqref{eq:flux_f_h}-\eqref{eq:flux_Bh}, has the following error estimate
		\begin{equation}
		\norma{(f-f_h)(t)}_{0,\Omega}^2+\norma{(\elec-\elec_h)(t)}_{0,\Omega_x}^2+\norma{(\magn-\magn_h)(t)}_{0,\Omega_x}^2\leq C h^{2k+1},\quad \forall t\in[0,T].
		\label{eq:dg_error_estimate}
		\end{equation}
	    Here the constant $C$ is independent of $h,$ but depends on the upper bounds of $\norma{\partial_t f}_{k+1,\Omega}$,$\norma{f}_{k+1,\Omega}$, $\left|f\right|_{1,\infty,\Omega}$, $\norma{\elec}_{1,\infty,\Omega_x}$, $\norma{\magn}_{1,\infty,\Omega_x}$, $\norma{\elec}_{k+1,\Omega_x}$, $\norma{\magn}_{k+1,\Omega_x}$ over the time interval $[0,T]$, and it also depends on the polynomial degree $k$, mesh parameters $\sigma_0$, $\sigma_x$ and $\sigma_v$, and domain parameters $L_x$ and $L_v$.
	\end{thm}
In this work, we also consider \eqref{eq:vm_system} when there is no magnetic field (i.e. when $\magn=0$). This reduced problem is called the VA system, and the DG discretizations would follow a similar discussion by setting $\magn_h=0$ in \eqref{eq:dg_system} at all times.

	\section{Smoothness-Increasing Accuracy-Conserving Filters}\label{sec:siac_filter}
	We extract the higher-order accuracy of the DG method solved over a uniform mesh contained in the negative-order norm by using the SIAC filter. This technique could also be applied over nonuniform meshes, however this would force us to compute the post-processing coefficients in each element in the mesh, increasing the computational complexity of the implementation \cite{curtis2008postprocessing}. This filter improves the order of accuracy by   reducing the spurious oscillations in the error. This is done by convolving the numerical approximation with a specially chosen kernel, 
	\begin{gather}
	    (f^*_h(\vecx,\vecv),\elec_h^*(\vecx),\magn_h^*(\vecx))=K_h^{2(k+1),k+1}\star (f_h, \elec_h,\magn_h)(\vecx,\vecv),
	    \label{eq:kernel_conv}
	\end{gather}
	where $(f_h^*,\elec_h^*,\magn_h^*)$ is the filtered solution, $(f_h,\elec_h,\magn_h)$ is an approximated solution computed at the final time, and $K_h^{2(k+1),k+1}$ is the convolution kernel. The kernel is translation-invariant and composed of a linear combination of B-splines of order $k+1$ obtained by convolving the characteristic function over the interval $(-\frac{1}{2},\frac{1}{2})$ with itself $k$ times and scaled by the uniform mesh size. Using B-splines makes this kernel computationally efficient, provided the mesh is uniform, as the kernel is translation invariant and is locally supported in at most $2k+2$ elements. The one-dimensional convolution kernel is of the  form:
	\begin{gather}
	    K_h^{2(k+1),k+1}(x)=\frac{1}{h}\sum_{\gamma=-k}^{k}c_{\gamma}^{2(k+1),k+1}\psi^{(k+1)}\left(\frac{x}{h}-\gamma\right).
	\end{gather}
	The weights of the B-splines, $c_{\gamma}^{2(k+1),k+1}$, are chosen so that accuracy is not destroyed (the kernel can reproduce polynomials of degree up to $2k$), i.e. $K_h^{2(k+1),k+1}\star p=p$ for $p=1,\,x,\,\cdots,x^{2k},$ see \cite{cockburn2003enhanced} for details. 
	
	For the general case, assume the mesh size is uniform in each direction, given arbitrary $(\vecx,\vecv)=(x_1,\cdots,x_{d_x},v_1,\cdots,v_{d_v})\in \R^{d_x+d_v}$, we set \begin{gather}
	    \psi^{(k+1)}(\vecx,\vecv)=\prod_{i=1}^{d_x}\psi^{(k+1)}(x_i)\prod_{j=1}^{d_v}\psi^{(k+1)}(v_j)
	\end{gather}
	The kernel for our case is of the form
	\begin{gather}
	    K_h^{2(k+1),k+1}(\vecx,\vecv)=\frac{1}{\left(\prod_{i=1}^{d_x}h_{x_i}\right)\left(\prod_{j=1}^{d_v}h_{v_j}\right)}\sum_{\gamma\in{\{-k,\ldots,k\}}^{d_x+d_v}}\mathbf{c}_{\gamma}^{2(k+1),k+1}\psi^{(k+1)}\left(\left(\frac{x_1}{h_{x_1}},\cdots,\frac{x_{d_x}}{h_{d_x}},\frac{v_1}{h_{v_1}},\cdots,\frac{v_{d_v}}{h_{d_v}}\right)-\gamma\right)
	\end{gather}
	where $h_{x_i}$ and $h_{v_i}$ denote the mesh size in $x_i$ and $v_i$ direction, resp. The success of the filter relies on the following results. 
	\begin{thm}(Bramble and Schatz \cite{bramble1977higher})\label{thmbs} For $T>0$, let $u=(f, \elec, \magn)$ be the exact solution of the problem \eqref{eq:vm_system}. Let $\Omega_0+2supp(K_h^{2(k+1),k+1}(\vecx,\vecv))\subset\subset \Omega$ and $U=(f_h, \elec_h, \magn_h)$ is any approximation to $u$, then 
	\begin{gather}
	\label{eqbs}
	    \norma{u(T)-K_h^{2(k+1),k+1}\star U}_{0,\Omega_0}\leq \frac{h^{2k+2}}{(2k+2)!}|u|_{2k+2,\Omega}+C_{\rm P}\sum_{|\lambda|\leq k+1}\norma{\partial_h^{\lambda}(u-U)}_{-(k+1),\Omega}.
	\end{gather}
	where $C_{\rm P}$ depends solely on $\Omega_0,\,\Omega,\,d_x,\,d_v,\,k, c_{\gamma}^{2(k+1),k+1}$, and it is independent of $h$. 
	\end{thm}

		In \eqref{eqbs}, we used the notation of the divided differences. 
		We define
	\begin{gather}
	    \partial_{h_{x_i}}w(\vecx,\vecv)=\frac{1}{h_{x_i}}\left(w\left(\vecx+\frac{1}{2}h_{x_i}\unitv_i,\vecv\right)-w\left(\vecx-\frac{1}{2}h_{x_i}\unitv_i,\vecv\right)\right),
	\end{gather}
    here $\unitv_i$ is the unit multi-index whose $i$-th component is $1$ and all others $0$. Analogously for velocity space variables $v_j$, the difference quotients are defined as
   	\begin{gather}
    	\partial_{h_{v_j}}w(\vecx,\vecv)=\frac{1}{h_{v_j}}\left(w\left(\vecx,\vecv+\frac{1}{2}h_{v_j}\unitv_j\right)-w\left(\vecx,\vecv-\frac{1}{2}h_{v_j}\unitv_j\right)\right),
    \end{gather}

     For any multi-index $\lambda=(\alpha_{x_1},\cdots,\alpha_{d_x},\beta_{v_1},\cdots,\beta_{d_v})$ we set $\alpha$-th order difference quotient to be
    \begin{gather}
        \partial_h^{\lambda}w(\vecx,\vecv)=(\partial_{h_{x_1}}^{\alpha_1}\cdots\partial_{h_{x_{d_x}}}^{\alpha_{d_x}})(\partial_{h_{v_1}}^{\beta_1}\cdots\partial_{h_{v_{d_v}}}^{\beta_{d_v}})w(\vecx,\vecv).
    \end{gather}

%
    \section{Superconvergent Error Estimates for the DG method}\label{sec:sc_dg_method}
In this section, we  prove the superconvergence error estimate in the negative norm of the DG solution for the VM system. In Section \ref{sec:prelim}, we review basic approximation and regularity properties. Section \ref{sec:dual} will construct the dual problem which is the key to our estimates. The main result and the proof will be given in Section \ref{sec:main}.
 	
 	\subsection{Preliminaries}
	\label{sec:prelim}
We summarize some of the standard approximation properties of the above discrete spaces, as well as some inverse inequalities \cite{ciarlet1979finite}. For any nonnegative integer $k$, Let $\Pi^k$ be the $L^2$ projection onto $\mathcal{G}^k_h$, and $\mathbf{\Pi}^m_x$ be the $L^2$ projection onto $\mathcal{U}^m_h$. We define $\zeta^g_h=\Pi^k g-g$ and $\zeta_h^{\U}=\mathbf{\Pi}_x^k\U-\U$, as the \emph{Projection errors} of $g$ and $\U$ respectively.
\begin{lem}(Approximation properties)\label{lem:approximation_lemma}
There exist a constant $C>0$, such that for any \mbox{$g\in H^{k+1}(\Omega)$} and $\mathbf{U}\in [H^{k+1}(\Omega)]^{d_x}$, the following hold:
\begin{align*}
	\norma{\zeta_h^{g}}_{0,K}+h_K\norma{\nabla_{\star}\zeta_h^{g}}_{0,K}+h_K^{1/2}\norma{\zeta_h^{g}}_{0,\partial K}&\leq Ch_{K}^{k+1}\norma{g}_{k+1,K},\quad \forall K\in\mathcal{T}_h\\
	\norma{\zeta_h^{\U}}_{0,K_x}+h_{K_x}\norma{\nabla_{\vecx}\times\zeta_h^{\U}}_{0,K_x}+h_{K_x}^{1/2}\norma{\zeta_h^{\U}}_{0,\partial K_x}&\leq Ch_{K_x}^{k+1}\norma{\U}_{k+1,K_x},\quad \forall K_x\in\mathcal{T}_h^x\\ 		\norma{\zeta_h^{\U}}_{0,\infty,K_x}&\leq Ch_{K_x}^{k+1}\norma{\U}_{k+1,\infty,K_x},\quad \forall K_x\in\mathcal{T}_h^x\\
\end{align*}
where the constant $C$ is independent of the mesh sizes $h_K$ and $h_{K_x}$, but depends on $k$ and the shape regularity parameters $\sigma_x$ and $\sigma_v$ of the mesh. Here $\star=x$ or $v$.
\end{lem}
\begin{lem}[Inverse inequality]\label{lem:inverse_inequalities}
There exists a constant $C>0$, such that for any $g\in P^k(K)$ or $P^k(K_x)\times P^k(K_v)$ with $K=(K_x\times K_v)\in\mathcal{T}_h$ , and for any $\U\in[P^k(K_x)]^{d_x}$, the following hold:
\begin{align*}
	\norma{\nabla_{\vecx}g}_{0,K}\leq Ch_{K_x}^{-1}\norma{g}_{0,K},&\quad\norma{\nabla_{\vecv}g}_{0,K}\leq Ch_{K_v}^{-1}\norma{g}_{0,K},\\
	\norma{\U}_{0,\infty,K_x}\leq Ch_{K_x}^{-d_x/2}\norma{\U}_{0,K_x},&\quad\norma{\U}_{0,\partial K_x}\leq Ch_{K_x}^{-1/2}\norma{\U}_{0,K_x},
\end{align*}
where the constant $C$ is independent of the mesh sizes $h_{K_x},\,h_{K_{v}}$, but depends on $k$ and the shape regularity parameters $\sigma_x$ and $\sigma_v$ of the mesh.
\end{lem}

To assist the proof, we also need a regularity result for a linear PDE system.
	\begin{lem}\label{lem:regularity_est} Consider the following system of equations with periodic boundary conditions in $\vecx$ and zero boundary condition in $\vecv$ for all $t\in[0,T]$:
	\begin{subequations}\label{eq:reg_equation}
		\begin{align}
			&\partial_t\varphi+\mathbf{A_1}(\vecx,\vecv,t)\cdot\nabla_{\vecx}\varphi+\mathbf{A_2}(\vecx,\vecv,t)\cdot\nabla_{\mathbf{v}}\varphi+\mathbf{A_3}(\vecx,\vecv,t)\cdot\F=0,\label{eq:gen_du_phi}\\
			&\partial_t \F=\rot\D+\int_{\Omega_v}g\nabla_{\vecv} \varphi\,d\vecv,\label{eq:gen_du_F}\\
			&\partial_t\D=-\rot\F-\int_{\Omega_v}g(\vecv\times\nabla_{\vecv}\varphi)\,d\vecv\label{eq:gen_du_D},
		\end{align} 
\end{subequations}
where the given functions $\mathbf{A_1}, \mathbf{A_2} \in W^{l+1,\infty}(\Omega)$ satisfy the divergence free constraint $\nabla_{\vecx} \cdot \mathbf{A_1}=0$ and $\nabla_{\vecv} \cdot \mathbf{A_2}=0.$  For any $l\geq 0$ and the fixed time $t$, the solution to \eqref{eq:reg_equation} satisfy the following estimate
\begin{equation}
	\norma{\varphi(\cdot,\cdot,t)}_{l,\Omega}^2+\norma{\F(\cdot,t)}_{l,\Omega_x}^2+\norma{\D(\cdot,t)}_{l,\Omega_x}^2\leq C\left[\norma{\varphi(\cdot,\cdot,0)}_{l,\Omega}^2+\norma{\F(\cdot,0)}_{l,\Omega_x}^2+\norma{\D(\cdot,0)}_{l,\Omega_x}^2\right].
	\label{eq:reg_phi}
\end{equation}
Here $C$ depends on $\norma{\mathbf{A_3}}_{L^\infty((0,T);W^{l,\infty}(\Omega))}$ and  $\norma{g}_{L^\infty((0,T);W^{l+1,\infty}(\Omega))}$.
	\end{lem}
	\begin{proof}
		See the appendix.
	\end{proof}
 
%

	\subsection{The dual problem}
	\label{sec:dual}
	
	In order to prove negative-order estimates for the system,  the key is to find the dual problem associated to \eqref{eq:vm_system}. We note that, for the nonlinear problem, the dual problem is not unique, see \cite{marchuk1998construction}. We construct the dual problem as follows: find functions $\varphi(\cdot,\cdot,t)$, $\mathbf{F}(\cdot,t)$ and $\mathbf{D}(\cdot,t)$ such that $\varphi(\cdot,\vecv,t)$ is periodic in all dimensions in space and $\varphi(\vecx,\cdot,t)$ vanishes in the boundary of the velocity region for all $t\in [0,T]$ and
	\begin{subequations}\label{eq:dual_system}
	\begin{align}
		&\partial_t\varphi+\vecv\cdot\nabla_{\vecx}\varphi+(\elec+\vecv\times\magn)\cdot\nabla_{\mathbf{v}}\varphi-\vecv\cdot\F=0\label{eq:dual_distribution}\\
		&\partial_t \F=\rot\D-\int_{\Omega_v}f\nabla_v\varphi\,d\vecv\label{eq:dual_elec},\\
		&\partial_t\D=-\rot\F+\int_{\Omega_v}f(\vecv\times\nabla_{\vecv}\varphi)\,d\vecv\label{eq:dual_magn}
	\end{align}
\end{subequations}
	with final time conditions $\varphi(\vecx,\vecv,T)=\Phi(\vecx),\,\F(\vecx,T)=\mathfrak{F}(\vecx)$ and $\D(\vecx,T)=\mathfrak{D}(\vecx)$, $\Phi\in C^{\infty}_0(\Omega)$ and $\mathfrak{D},\mathfrak{F}\in \left[C^{\infty}_{0}(\Omega_x)\right]^{d_x}$.
	
	Notice that by multiplying $(\varphi,\F,\D)$ on both sides of \eqref{eq:vm_main_eq}-\eqref{eq:temp_elec_magn}, and multiplying by $(f,\elec,\magn)$ on both sides of \eqref{eq:dual_distribution}-\eqref{eq:dual_magn}, and then summing up and integrating over velocity and physical space, we obtain
%
     \begin{align*}
     	&\int_{\Omega}\partial_t(f\varphi)\,d\vecx\,d\vecv+\int_{\Omega}\nabla_x\cdot(f\varphi\vecv)\,d\vecx\,d\vecv+\int_{\Omega}\nabla_{\mathbf{v}}\cdot\left(f\varphi(\elec+\vecv\times\magn)\right)\,d\vecx\,d\vecv-\int_{\Omega}f\vecv\cdot\F\,d\vecx\,d\vecv=0,\\
     	&\int_{\Omega_x}\partial_t(\elec\cdot\F+\magn\cdot\D)\,d\vecx+\int_{\Omega}f\vecv\cdot\F\,d\vecx\,d\vecv\\&=\int_{\Omega_x}\nabla_{\vecx}\cdot(\magn\times\F)\,d\vecx+\int_{\Omega_x}\nabla_{\vecx}\cdot(\elec\times\D)\,d\vecx-\int_{\Omega}f(\elec+\vecv\times\magn)\cdot\nabla_{\mathbf{v}}\varphi\,d\vecx d\vecv,
     \end{align*}
where we used the identities 
      \begin{gather*}
      	\nabla_{\star}(\phi\U)=\phi\nabla_{\star}\cdot\U+\U\cdot\nabla_{\star}\phi,\\
      	\nabla_{\star}\cdot(\U\times\mathbf{W})=\mathbf{W}\cdot(\nabla_{\star}\times\mathbf{U})-\U\cdot(\nabla_{\star}\times\mathbf{W}),
      \end{gather*}
    for scalar functions $\phi$ and vector functions $\U$ and $\mathbf{W}$ and the fact that $\nabla_{\vecv}\cdot(\elec+\vecv\times\magn)=\mathbf{0}$. 
    
    By  adding all equations above and using   boundary conditions, we arrive at
	\begin{equation}
		\frac{d}{dt}[(f,\varphi)_{\Omega}+(\elec,\F)_{\Omega_x}+(\magn,\D)_{\Omega_x}]+\mathcal{F}(f,\elec,\magn;\varphi)=0,\label{eq:important_derivative}
	\end{equation}
	where 
	\begin{equation}
		\mathcal{F}(f,\elec,\magn;\varphi)=\int_{\Omega}f(\elec+\vecv\times\magn)\cdot\nabla_{\mathbf{v}}\varphi\,d\vecx d\vecv.
		\label{eq:bigf_def}
	\end{equation}
	
	
	\subsection{The main result}
	\label{sec:main}
	
	In this part, we give our main theorem on the negative-norm of the error for the DG solutions. Note that superconvergence of the negative norm of the solution itself is not sufficient in proving high order convergence of the post-processed solution according to Theorem \ref{thmbs}. However, it is a necessary first step. As shown in \cite{meng2017discontinuous}, it is highly nontrivial to prove superconvergence of the divided difference of the solution for nonlinear problems, we will leave this to explore in our future work.
	\begin{thm}\label{thm:zero_divided_difference} If $(f_h,\elec_h,\magn_h)$ is a solution to \eqref{eq:scheme_f}-\eqref{eq:scheme_fields} with the numerical initial condition $f_h=\Pi^k f$ and $\elec_h=\mathbf{\Pi}_x^k\elec, \, \magn_h=\mathbf{\Pi}_x^k\magn$ and $k\geq (d_x+d_v)/2$, then 
		\begin{equation*}
			\norma{(f-f_h,\elec-\elec_h,\magn-\magn_h)}_{-(k+1),\Omega}\leq C h^{2k+1/2},
		\end{equation*}
	where $C$ is a constant independent of $h$ and depends on the upper bounds of $\norma{\partial_t f}_{k+2,\Omega}$,$\norma{f}_{k+2,\Omega}$, $\left|f\right|_{1,\infty,\Omega}$, $\norma{\elec}_{1,\infty,\Omega_x}$, $\norma{\magn}_{1,\infty,\Omega_x}$, $\norma{\elec}_{k+2,\Omega_x}$, $\norma{\magn}_{k+2,\Omega_x}$ over the time interval $[0,T]$, and it also depends on the polynomial degree $k$, mesh parameters $\sigma_0$, $\sigma_x$ and $\sigma_v$, and domain parameters $L_x$ and $L_v$.
 	\end{thm}
\begin{proof}
	We define $e_h^f=f-f_h=\varepsilon_h^f-\zeta_h^f$, where $\varepsilon_h^f=\Pi^k f-f_h$ and $\zeta^{f}_h$ is defined just as in Section \ref{sec:prelim}. Analogously  $\varepsilon_h^{\elec}=\mathbf{\Pi}_x^k\elec-\elec_h$, $\varepsilon_h^{\magn}=\mathbf{\Pi}_x^k\magn-\magn_h$, then  $e_h^{\elec}=\elec-\elec_h=\varepsilon_h^{\elec}-\zeta_h^{\elec}$ and $e_h^{\magn}=\magn-\magn_h=\varepsilon_h^{\magn}-\zeta_h^{\magn}$.
	We follow the ideas in \cite{cockburn2003enhanced}. For any $\Phi\in C_0^{\infty}(\Omega),\mathfrak{F},\mathfrak{D}\in[C_0^{\infty}(\Omega_x)]^{d_x}$, we   estimate the term 
%
%
	\begin{align*}
		&(e_h^f(T),\Phi)_{\Omega}+(e_h^{\elec}(T),\mathfrak{F})_{\Omega_x}+(e_h^{\magn}(T),\mathfrak{D})_{\Omega_x}\\
		=&(e_h^f(T),\varphi(T))_{\Omega}+(e_h^{\elec}(T),\F(T))_{\Omega_x}+(e_h^{\magn}(T),\D(T))_{\Omega_x}\\
		=&(f(T),\varphi(T))_{\Omega}+(\elec(T),\F(T))_{\Omega_x}+(\magn(T),\D(T))_{\Omega_x}\\
		&-[(f_h(T),\varphi(T))_{\Omega}+(\elec_h(T),\F(T))_{\Omega_x}+(\magn_h(T),\D(T))_{\Omega_x}]\\
		=&(f_0,\varphi(0))_{\Omega}+(\elec_0,\F(0))_{\Omega_x}+(\magn_0,\D(0))_{\Omega_x}-\int_{0}^{T}\mathcal{F}(f,\elec,\magn;\varphi)\,d\tau\\
		&-(f_h(0),\varphi(0))_{\Omega}-(\elec_h(0),\F(0))_{\Omega_x}-(\magn_h(0),\D(0))_{\Omega_x}-\int_0^T\frac{d}{dt}[(f_h,\varphi)_{\Omega}+(\elec_h,\F)_{\Omega_x}+(\magn_h,\D)_{\Omega_x}]\,d\tau\\
		=&-\left[(\zeta_h^{f_0},\varphi(0))_{\Omega}+(\zeta_h^{\elec_0},\F(0))_{\Omega_x}+(\zeta_h^{\magn},\D(0))_{\Omega_x}\right]-\int_0^T((f_h)_t,\varphi)_{\Omega}+((\elec_h)_t,\F)_{\Omega_x}+((\magn_h)_t,\D)_{\Omega_x}\,d\tau\\
		&-\int_0^T(f_h,\varphi_t)_{\Omega}+(\elec_h,\F_t)_{\Omega_x}+(\magn_h,\D_t)_{\Omega_x}+\mathcal{F}(f,\elec,\magn;\varphi)\,d\tau,
	\end{align*}
 where   for the first equality we used \eqref{eq:important_derivative}, and the numerical initial condition is used in the last equality.
    Notice that for any $\chi\in\mathcal{G}_h^k$, $\xi,\eta\in\mathcal{U}_h^k$
	\begin{align*}
		&\int_0^T((f_h)_t,\varphi)_{\Omega}+((\elec_h)_t,\F)_{\Omega_x}+((\magn_h)_t,\D)_{\Omega_x}\,d\tau\\
		=&	\int_0^T((f_h)_t,\varphi-\chi)_{\Omega}\,d\tau+\int_0^T((f_h)_t,\chi)_{\Omega}\,d\tau+\int_{0}^{T}((\elec_h)_t,\F-\xi)_{\Omega_x}+((\magn_h)_t,\D-\eta)_{\Omega_x}\,d\tau\\
		&+\int_{0}^{T}((\elec_h)_t,\xi)_{\Omega_x}+((\magn_h)_t,\eta)_{\Omega_x}\,d\tau\\
		=&\int_0^T((f_h)_t,\varphi-\chi)_{\Omega}\,d\tau-\int_0^T a_h(f_h,\elec_h,\magn_h;\chi)\,d\tau+\int_{0}^{T}((\elec_h)_t,\F-\xi)_{\Omega_x}+((\magn_h)_t,\D-\eta)_{\Omega_x}\,d\tau\\
		&-\int_{0}^{T}b_h(\elec_h,\magn_h;\xi,\eta)-l_h(\curre_h,\xi)\,d\tau\\
		=&\int_0^T((f_h)_t,\varphi-\chi)_{\Omega}+a_h(f_h,\elec_h,\magn_h;\varphi-\chi)\,d\tau+\int_{0}^{T}((\elec_h)_t,\F-\xi)_{\Omega_x}+((\magn_h)_t,\D-\eta)_{\Omega_x}\,d\tau\\
		&+\int_{0}^{T}b_h(\elec_h,\magn_h;\F-\xi,\D-\eta)-l_h(\curre_h,\F-\xi)\,d\tau-\int_0^T a_h(f_h,\elec_h,\magn_h;\varphi)\,d\tau\\
		&-\int_0^T b_h(\elec_h,\magn_h;\F,\D)-l_h(\curre_h,\F)\,d\tau.\\
	\end{align*}
	After this calculation we can conclude that 
	\begin{equation}
		(e_h^f(T),\Phi)_{\Omega}+(e_h^{\elec}(T),\mathfrak{F})_{\Omega_x}+(e_h^{\magn}(T),\mathfrak{D})_{\Omega_x}=\Theta_{\rm M}+\Theta_{\rm N}+\Theta_{\rm C},
	\end{equation}
	where 
	\begin{subequations}
		\begin{align*}
			\Theta_{\rm M}&=-\left[(\zeta_h^{f_0},\varphi(0))_{\Omega}+(\zeta_h^{\elec_0},\F(0))_{\Omega_x}+(\zeta_h^{\magn_0},\D(0))_{\Omega_x}\right],\\
			\Theta_{\rm N}&=-\int_0^T((f_h)_t,\varphi-\chi)_{\Omega}+a_h(f_h,\elec_h,\magn_h;\varphi-\chi)\,d\tau\\
			&-\int_0^T((\elec_h)_t,\F-\xi)_{\Omega_x}+((\magn_h)_t,\D-\eta)_{\Omega_x}+b_h(\elec_h,\magn_h;\F-\xi,\D-\eta)-l_h(\curre_h,\F-\xi)\,d\tau,\\
			\Theta_{\rm C}&=-\int_0^T(f_h,\varphi_t)_{\Omega}-a_h(f_h,\elec_h,\magn_h;\varphi)\,d\tau\\
			&-\int_0^T(\elec_h,\F_t)_{\Omega_x}+(\magn_h,\D_t)_{\Omega_x}-b_h(\elec_h,\magn_h;\F,\D)+l_h(\curre_h,\F)\,d\tau-\int_{0}^{T}\mathcal{F}(f,\elec,\magn;\varphi)\,d\tau.
		\end{align*}
	\end{subequations}
	
	In the following we will estimate $\Theta_{\rm M},\,\Theta_{\rm N}$ and $\Theta_{\rm C}$. 
	\begin{lem}[Projection Estimate]\label{lem:proj_estimate} $\Theta_{\rm M}$ satisfies
		\begin{equation}
			|\Theta_{\rm M}|\leq C h^{2k+2}\sqrt{\norma{\varphi(0)}_{k+1,\Omega}^2+\norma{\F(0)}_{k+1,\Omega_x}^2+\norma{\D(0)}_{k+1,\Omega_x}^2}
		\end{equation}
		where $C$ depends on $\norma{f_0}_{k+1,\Omega},\,\norma{\elec_0}_{k+1,\Omega_x}$ and $\norma{\magn_0}_{k+1,\Omega_x}.$
	\end{lem}

	\begin{proof} See appendix.
	\end{proof}
	For the second term, we have the following result:
	\begin{lem}[Residual] \label{lem:residual} Let $\chi=\Pi^kf, \xi=\mathbf{\Pi}_x^k\F, \eta=\mathbf{\Pi}_x^k\D$, we have 
			\begin{align*}
			|\Theta_{\rm N}|&\leq C h^{2k+1/2}\left[\int_{0}^{T}\norma{\varphi}_{k+1,\Omega}^2+\norma{\F}_{k+1,\Omega_x}^2+\norma{\D}_{k+1,\Omega_x}^2\,dt\right]^{1/2}
		\end{align*}
		where $C$  depends on the upper bounds of $\norma{f}_{k+2,\Omega}$, $\norma{f}_{1,\infty,\Omega}$, $\norma{\elec}_{0,\infty,\Omega_x}$, $\norma{\magn}_{0,\infty,\Omega_x}$, $\norma{\elec}_{k+2,\Omega_x}$, $\norma{\magn}_{k+2,\Omega_x}$ over the time interval $[0,T]$, and it also depends on the polynomial degree $k$, mesh parameters $\sigma_0$, $\sigma_x$ and $\sigma_v$, and domain parameters $L_x$ and $L_v$.
	\end{lem}
	\begin{proof}  See appendix. \end{proof}


	Lastly, we need to estimate the third term, $\Theta_{\rm C}$.
	\begin{lem}[Consistency]\label{lem:consistency} We have
		\begin{equation} 
			|\Theta_{\rm C}|\leq C h^{2k+1}\left[\int_{0}^{T}\norma{\varphi}_{k+1,\Omega}^2 \,dt\right]^{1/2}		\end{equation}
		where $C$   depends on the upper bounds of $\norma{\partial_t f}_{k+1,\Omega}$,$\norma{f}_{k+1,\Omega}$, $\left|f\right|_{1,\infty,\Omega}$, $\norma{\elec}_{1,\infty,\Omega_x}$, $\norma{\magn}_{1,\infty,\Omega_x}$, $\norma{\elec}_{k+1,\Omega_x}$, $\norma{\magn}_{k+1,\Omega_x}$ over the time interval $[0,T]$, and it also depends on the polynomial degree $k$, mesh parameters $\sigma_0$, $\sigma_x$ and $\sigma_v$, and domain parameters $L_x$ and $L_v$.
	\end{lem}
	\begin{proof}
	See appendix.
	\end{proof}
%
	It is easy to transform the dual problem \eqref{eq:dual_system} to an initial value problem \eqref{eq:reg_equation} by changing time $t'=T-t$. Then using Lemma \ref{lem:regularity_est}, where $\mathbf{A_1}(\vecx,\vecv,t)=-\vecv$, $\mathbf{A_2}(\vecx,\vecv,t)=-(\elec+\vecv\times\magn)$, $\mathbf{A_3}(\vecx,\vecv,t)=\vecv$, $g=f$ and $l=k+1$,   
	\begin{align}
			\norma{\varphi}_{k+1,\Omega}^2+\norma{\F}_{k+1,\Omega_x}^2+\norma{\D}_{k+1,\Omega_x}^2&\leq  C [\norma{\Phi}_{k+1,\Omega}^2+\norma{\mathfrak{F}}_{k+1,\Omega_x}^2+\norma{\mathfrak{D}}_{k+1,\Omega_x}^2]
	\end{align}
	where $C$ depends on $\norma{f}_{L^\infty((0,T);W^{k+2,\infty}(\Omega))}.$
    Then an application of Theorem \ref{thm:main_approx_result}   gives us
    \begin{equation}
    	|(e_h^f(T),\Phi)_{\Omega}+(e_h^{\elec}(T),\mathfrak{F})_{\Omega_x}+(e_h^{\magn}(T),\mathfrak{D})_{\Omega_x}|\leq C h^{2k+1/2}\sqrt{\norma{\Phi}_{k+1,\Omega}^2+\norma{\mathfrak{F}}_{k+1,\Omega_x}^2+\norma{\mathfrak{D}}_{k+1,\Omega_x}^2}
    	\label{eq:importante_estimate_1}
    \end{equation}
  	
 	Therefore the estimate for the zero-divided difference negative-order norm is given by
 	\begin{gather*}
 		\norma{(f-f_h,\elec-\elec_h,\magn-\magn_h)}_{-(k+1),\Omega}\\=\sup_{\phi\in C_0^{\infty}(\Omega),\mathfrak{F},\mathfrak{D}\in[C^{\infty}(\Omega_x)]^{d_x}}\frac{(f-f_h,\Phi)_{\Omega}+(\elec-\elec_h,\mathfrak{F})_{\Omega_x}+(\magn-\magn_h,\mathfrak{D})_{\Omega_x}}{\sqrt{\norma{\Phi}_{k+1,\Omega}^2+\norma{\mathfrak{F}}_{k+1,\Omega_x}^2+\norma{\mathfrak{D}}_{k+1,\Omega_x}^2}} 
 		\leq C h^{2k+1/2}.
 	\end{gather*}
 	 	
\end{proof}

\section{Numerical Experiments}\label{sec:num_experiments}
In this section, we validate our theoretical results using several numerical tests. In particular, we want to demonstrate the performance of the post-processing technique for the VA system and the VM system. 
We heavily use the fact that the VM (VA) system is time reversible to provide quantitative measurements of the errors. In particular, let $f(\vecx,\vecv,0)$, $\elec(\vecx,0)$, $\magn(\vecx,0)$ denote the initial conditions and $f(\vecx,\vecv,T)$, $\elec(\vecx,T)$, $\magn(\vecx,T)$ be the solution of the VM system at $t=T$. If we choose $f(\vecx,-\vecv,T)$, $\elec(\vecx,T)$, $-\magn(\vecx,T)$ as the initial condition at $t=0$, then evolving the VM system to $t=T,$ we will recover  $f(\vecx,-\vecv,0)$, $\elec(\vecx,0)$, $-\magn(\vecx,0)$.

\subsection{Vlasov-Amp\'{e}re examples}

We consider two classical benchmark examples.
\begin{itemize}
     	\item Landau damping:
     	\begin{equation}
     		f(x,v,0)=f_M(v)(1+A\cos(kx)),\quad x\in[0,L],\,v\in[-V_c,V_c],
     		\label{eq:landau_damping}
     	\end{equation}
     where $A=0.5$, $k=0.5$, $L=4\pi$, $V_c=6\pi$, and $f_M(v)=\frac{1}{\sqrt{2\pi}}e^{-v^2/2}$. 
          	\item Two-stream instability:
     \begin{equation}
     	f(x,v,0)=f_{TS}(v)(1+A\cos(kx)),\quad x\in[0,L],\,v\in[-V_c,V_c],
     	\label{eq:two_stream}
     \end{equation}
     where $A=0.05$, $k=0.5$, $L=4\pi$, $V_c=6\pi$, and \mbox{$f_{TS}(v)=\frac{1}{\sqrt{2\pi}}v^2e^{-v^2/2}$}.
\end{itemize}
Notice that in both examples we have taken $V_c$ to be larger than the usual values in the literature in order to completely eliminate the boundary effects and accurately reflect the accuracy enhancement property. 

In Tables \ref{tab:lamdau_damping_error}, we run the VA system with initial condition from Landau damping to $T=1$ and then back to $T=0$ and then we apply the SIAC filter, and compare it with the initial conditions. 
 We use the third order TVD-RK method as the time integrator \cite{gottlieb1998total}. To make sure the spatial error dominates, we take $\Delta t=\mathrm{CFL}/(V_c/\Delta x+\elec_{\mathrm{max}}/\Delta v)$ for $\mathbb{P}^1$, $\elec_{\mathrm{max}}$ denotes the maximum value of $\elec(\cdot,T)$ in $\Omega_x$, for $\mathbb{P}^2$ we take $\Delta t=\mathrm{CFL}/(V_c/(\Delta x)^{5/3}+E_{max}/(\Delta v)^{5/3})$, and $\Delta t=\mathrm{CFL}/(V_c/(\Delta x)^{7/3}+E_{max}/(\Delta v)^{7/3})$ for $\mathbb{P}^3$. For $\mathbb{P}^1$ and $\mathbb{P}^3$ we take the $\mathrm{CFL}=0.1$, and we take the $\mathrm{CFL}=0.2$ for $\mathbb{P}^2$. 
 From the table, we observe $(k+1)$-th order of convergence for the DG solution before post-processing for both $f$ and $\elec$. We can clearly see that we improve the order of the error  to at least $O(h^{2k+1/2})$ after post-processing. 

In Figure \ref{fig:landau_error_comparision} we plot the errors of the numerical solution before and after post-processing for $\mathbb{P}^1$ and using $128\times 128$ elements. We can see that the errors before post-processing are highly oscillatory, and that the post-processing smooths out the error   and greatly reduces its magnitude. 
In Figure \ref{fig:elec_error_ld}, we plot the errors of the approximations for $\elec$ obtained when solving using a $128\times 128$ mesh with $\mathbb{P}^1$ and $32\times 32$ mesh with $\mathbb{P}^3$. We can clearly see that the errors before post-processing are highly oscillatory, and the post-processing gets rid of the oscillations and dramatically reduces the magnitude of the error.  
 Another point that we want to make is the following: if we look at Table \ref{tab:lamdau_damping_error}, for $k=2$ and a mesh of $64\times 64,$ the $L^2$-errors before and after post-processing are similar in magnitude. However, if we look at Figure \ref{fig:ld_inf_comp} which plots the absolute value of the error in $f$ in this case, we can clearly see that the $L^\infty$-norm  of the error of the filtered solution is much smaller than the unfiltered solution. Therefore, by removing the spurious oscillations, even if the $L^2$-error is comparable, the $L^\infty$ error is further reduced by the post-processor. This is probably due to the high oscillatory nature of the solution.
 

\begin{table}[!htbp]
\centering
\begin{tabular}{|ccccccccc|}
\hline
\multicolumn{1}{|c|}{}                & \multicolumn{4}{c|}{Before post-processing}                                                                                        & \multicolumn{4}{c|}{After post-processing}                                                                        \\ \hline
\multicolumn{1}{|c|}{mesh}            & \multicolumn{1}{c|}{error $f$} & \multicolumn{1}{c|}{order} & \multicolumn{1}{c|}{error $\mathbf{E}$} & \multicolumn{1}{c|}{order} & \multicolumn{1}{c|}{error $f^*$} & \multicolumn{1}{c|}{order} & \multicolumn{1}{c|}{error $\mathbf{E}^*$} & order \\ \hline
\multicolumn{9}{|c|}{$\mathbb{P}^1$}                                                                                                                                                                                                                                                           \\ \hline
\multicolumn{1}{|c|}{$16\times 16$}   & \multicolumn{1}{c|}{1.42E-02}  & \multicolumn{1}{c|}{-}     & \multicolumn{1}{c|}{1.19E-02}           & \multicolumn{1}{c|}{-}     & \multicolumn{1}{c|}{2.28E-02}    & \multicolumn{1}{c|}{-}     & \multicolumn{1}{c|}{1.04E-02}             & -     \\
\multicolumn{1}{|c|}{$32\times 32$}   & \multicolumn{1}{c|}{6.22E-03}  & \multicolumn{1}{c|}{1.19}  & \multicolumn{1}{c|}{3.16E-03}           & \multicolumn{1}{c|}{1.91}  & \multicolumn{1}{c|}{6.16E-03}    & \multicolumn{1}{c|}{1.89}  & \multicolumn{1}{c|}{2.84E-03}             & 1.88  \\
\multicolumn{1}{|c|}{$64\times 64$}   & \multicolumn{1}{c|}{1.59E-03}  & \multicolumn{1}{c|}{1.97}  & \multicolumn{1}{c|}{5.65E-04}           & \multicolumn{1}{c|}{2.48}  & \multicolumn{1}{c|}{8.74E-04}    & \multicolumn{1}{c|}{2.82}  & \multicolumn{1}{c|}{4.36E-04}             & 2.70  \\
\multicolumn{1}{|c|}{$128\times 128$} & \multicolumn{1}{c|}{4.08E-04}  & \multicolumn{1}{c|}{1.96}  & \multicolumn{1}{c|}{1.12E-04}           & \multicolumn{1}{c|}{2.33}  & \multicolumn{1}{c|}{1.10E-04}    & \multicolumn{1}{c|}{2.99}  & \multicolumn{1}{c|}{6.31E-05}             & 2.79  \\
\multicolumn{1}{|c|}{$256\times 256$} & \multicolumn{1}{c|}{1.03E-04}  & \multicolumn{1}{c|}{1.98}  & \multicolumn{1}{c|}{2.51E-05}           & \multicolumn{1}{c|}{2.16}  & \multicolumn{1}{c|}{1.37E-05}    & \multicolumn{1}{c|}{3.00}  & \multicolumn{1}{c|}{9.01E-06}             & 2.81  \\
\multicolumn{1}{|c|}{$512\times 512$} & \multicolumn{1}{c|}{2.60E-05}  & \multicolumn{1}{c|}{1.99}  & \multicolumn{1}{c|}{6.14E-06}           & \multicolumn{1}{c|}{2.03}  & \multicolumn{1}{c|}{1.71E-06}    & \multicolumn{1}{c|}{3.00}  & \multicolumn{1}{c|}{1.71E-06}             & 2.39  \\ \hline
\multicolumn{9}{|c|}{$\mathbb{P}^2$}                                                                                                                                                                                                                                                           \\ \hline
\multicolumn{1}{|c|}{$16\times 16$}   & \multicolumn{1}{c|}{7.08E-03}  & \multicolumn{1}{c|}{-}     & \multicolumn{1}{c|}{1.97E-03}           & \multicolumn{1}{c|}{-}     & \multicolumn{1}{c|}{2.09E-02}    & \multicolumn{1}{c|}{}      & \multicolumn{1}{c|}{1.88E-03}             & -     \\
\multicolumn{1}{|c|}{$32\times 32$}   & \multicolumn{1}{c|}{1.08E-03}  & \multicolumn{1}{c|}{2.71}  & \multicolumn{1}{c|}{1.13E-04}           & \multicolumn{1}{c|}{4.12}  & \multicolumn{1}{c|}{2.87E-03}    & \multicolumn{1}{c|}{2.87}  & \multicolumn{1}{c|}{1.08E-04}             & 4.12  \\
\multicolumn{1}{|c|}{$64\times 64$}   & \multicolumn{1}{c|}{1.35E-04}  & \multicolumn{1}{c|}{3.00}  & \multicolumn{1}{c|}{6.62E-06}           & \multicolumn{1}{c|}{4.10}  & \multicolumn{1}{c|}{1.20E-04}    & \multicolumn{1}{c|}{4.58}  & \multicolumn{1}{c|}{5.15E-06}             & 4.39  \\
\multicolumn{1}{|c|}{$128\times 128$} & \multicolumn{1}{c|}{1.63E-05}  & \multicolumn{1}{c|}{3.04}  & \multicolumn{1}{c|}{5.59E-07}           & \multicolumn{1}{c|}{3.57}  & \multicolumn{1}{c|}{2.70E-06}    & \multicolumn{1}{c|}{5.47}  & \multicolumn{1}{c|}{2.04E-07}             & 4.66  \\
\multicolumn{1}{|c|}{$256\times 256$} & \multicolumn{1}{c|}{2.01E-06}  & \multicolumn{1}{c|}{3.03}  & \multicolumn{1}{c|}{6.57E-08}           & \multicolumn{1}{c|}{3.09}  & \multicolumn{1}{c|}{5.29E-08}    & \multicolumn{1}{c|}{5.67}  & \multicolumn{1}{c|}{5.75E-09}             & 5.15  \\ \hline
\multicolumn{9}{|c|}{$\mathbb{P}^3$}                                                                                                                                                                                                                                                           \\ \hline
\multicolumn{1}{|c|}{$16\times 16$}   & \multicolumn{1}{c|}{1.73E-03}  & \multicolumn{1}{c|}{-}     & \multicolumn{1}{c|}{2.19E-04}           & \multicolumn{1}{c|}{-}     & \multicolumn{1}{c|}{2.16E-02}    & \multicolumn{1}{c|}{-}     & \multicolumn{1}{c|}{9.71E-05}             & -     \\
\multicolumn{1}{|c|}{$32\times 32$}   & \multicolumn{1}{c|}{1.52E-04}  & \multicolumn{1}{c|}{3.51} & \multicolumn{1}{c|}{7.18E-06}           & \multicolumn{1}{c|}{4.93} & \multicolumn{1}{c|}{2.60E-03}    & \multicolumn{1}{c|}{3.05}  & \multicolumn{1}{c|}{3.09E-06}             & 4.97 \\
\multicolumn{1}{|c|}{$64\times 64$}   & \multicolumn{1}{c|}{1.06E-05}  & \multicolumn{1}{c|}{3.84}  & \multicolumn{1}{c|}{1.30E-07}           & \multicolumn{1}{c|}{5.79}  & \multicolumn{1}{c|}{5.65E-05}    & \multicolumn{1}{c|}{5.52}  & \multicolumn{1}{c|}{7.52E-08}             & 5.36  \\
\multicolumn{1}{|c|}{$128\times 128$} & \multicolumn{1}{c|}{6.45E-07}  & \multicolumn{1}{c|}{4.04}  & \multicolumn{1}{c|}{3.42E-09}           & \multicolumn{1}{c|}{5.25}  & \multicolumn{1}{c|}{3.95E-07}    & \multicolumn{1}{c|}{7.16}  & \multicolumn{1}{c|}{8.24E-10}             & 6.51  \\
\hline
\end{tabular}
\caption{$L^2$ errors for the numerical solution and the post-processed solution for Landau Damping.}
\label{tab:lamdau_damping_error}
\end{table}

\begin{figure}[!htbp]
	\centering
	\includegraphics[width=\textwidth]{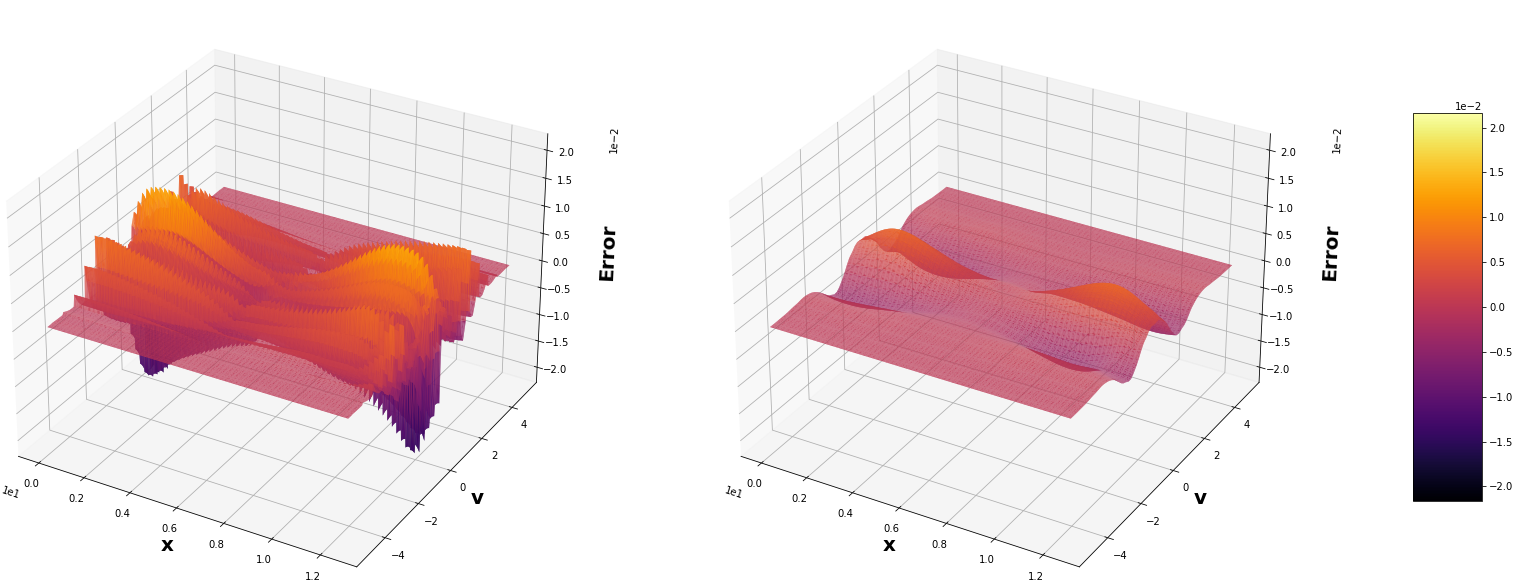}
	\caption{Errors for $f$ before (on the left) and after post-processing (on the right) for $128\times 128$ elements and $\mathbb{P}^1.$ Landau damping.}
	\label{fig:landau_error_comparision}
\end{figure}

\begin{figure}[!htbp]
	\centering
	\begin{subfigure}[b]{0.48\textwidth}
		\centering
		\includegraphics[width=\textwidth]{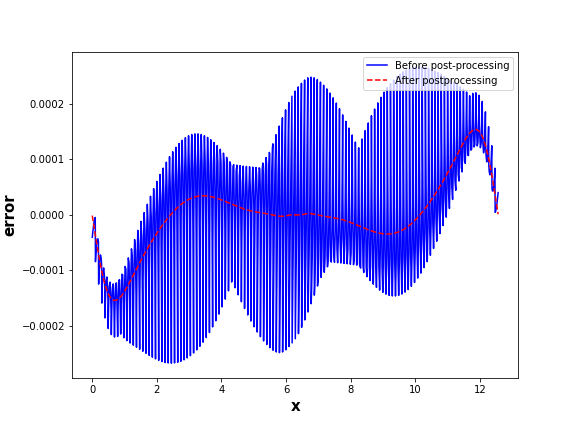}
		\caption{{\small $128\times 128$ and $\mathbb{P}^1$}}  
	\end{subfigure}
	\begin{subfigure}[b]{0.48\textwidth}  
		\centering 
		\includegraphics[width=\textwidth]{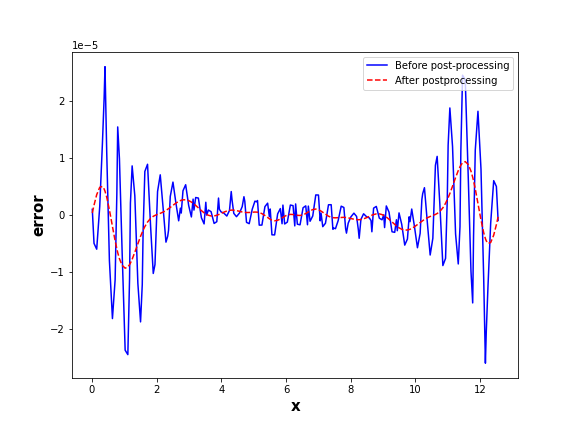}
		\caption{{\small $32\times 32$ and $\mathbb{P}^3$}}  
	\end{subfigure}
	\caption{\small Errors before (solid line) and after post-processing (dashed line) for $\elec$ for different mesh sizes and $\mathbb{P}^k$.  Landau damping. $T=2$.} 
	\label{fig:elec_error_ld} 
\end{figure}   

\begin{figure}[!htbp]
	\centering
	\includegraphics[width=\textwidth]{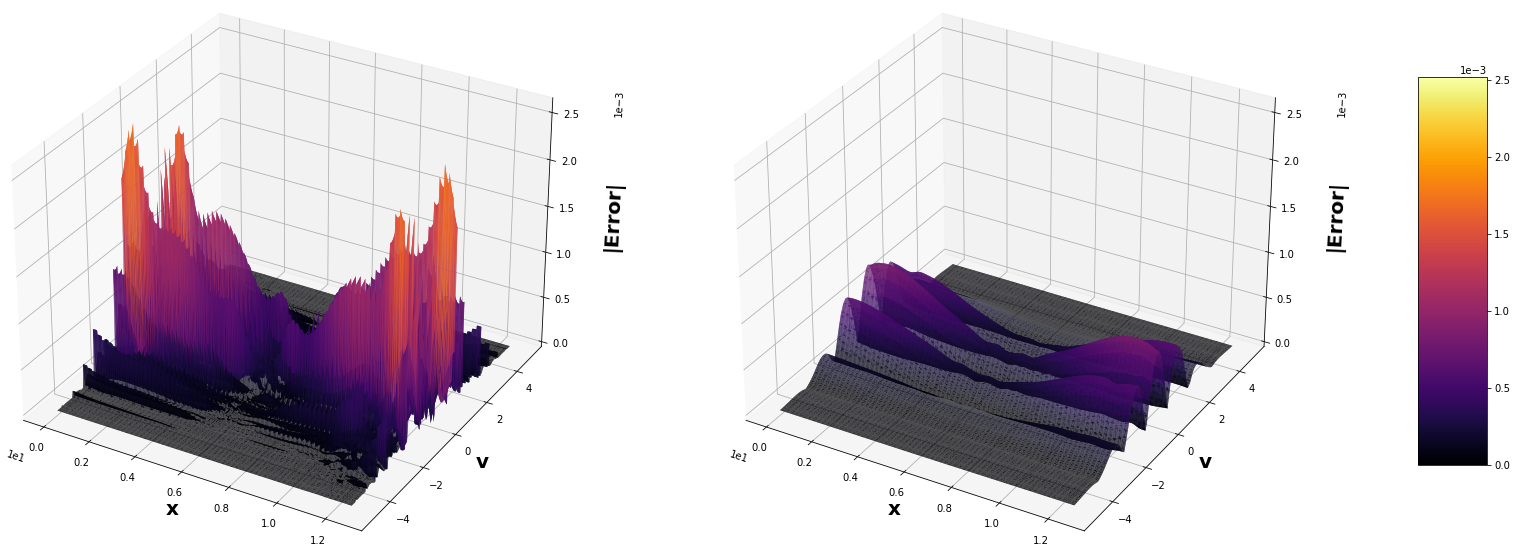}
	\caption{Absolute value of errors  for $f$ before (on the left) and after post-processing (on the right) for $64 \times 64$ elements and $\mathbb{P}^2.$ Landau damping.}	\label{fig:ld_inf_comp}
\end{figure}

Now we provide plots comparing the solution profile before and after post-processing for a longer computational time. To compute those plots, we   use a third-order Runge-Kutta method with $\Delta t=\mathrm{CFL}/(V_c/\Delta x+E_{max}/\Delta v))$ and $\mathrm{CFL}=0.1.$  In Figures \ref{fig:con_ld_k1_t10} to \ref{fig:con_ts_k2_t20}, we show a comparison of contour plots of the numerical solution for $f$ before and after post-processing with different mesh size and  $k=1,2.$   There is visible improvement of the resolution of the solution, particularly for   $k=1$. We also plot
 the macroscopic quantities: particle density and current density of the results for Landau-Damping with $k=1$ and $T=10$ on a $32 \times 32$ mesh before and after post-processing in Figure \ref{fig:pp-density}. It is clear that the spurious oscillations in those macroscopic quantities are removed by the filter.
 


\begin{figure}[!htbp]
	\centering
	\begin{subfigure}[b]{0.7\textwidth}
		\centering
		\includegraphics[width=\textwidth]{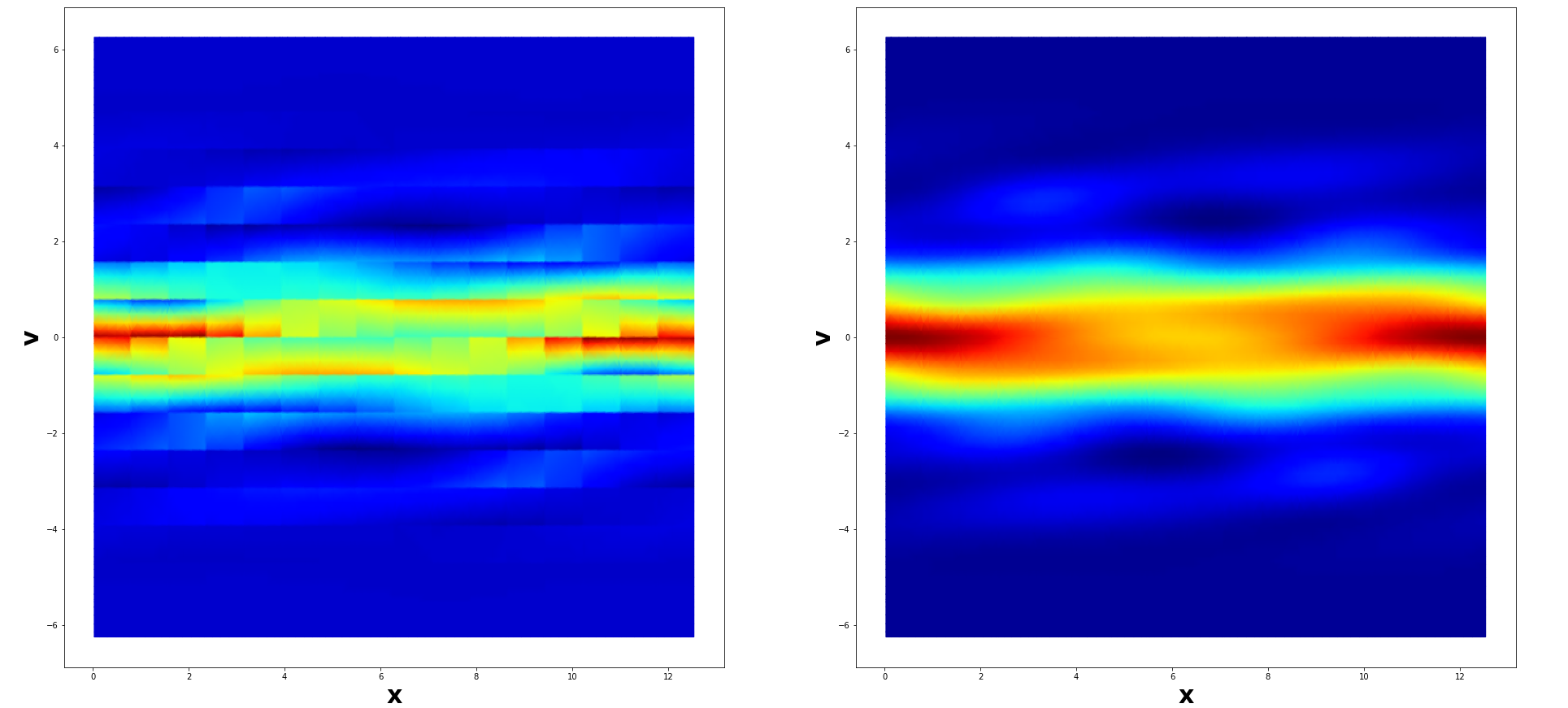}
		\caption{{\small $16\times 16$}}    
		\label{fig:con_ld_k1_t10_16}
	\end{subfigure}
	\vskip\baselineskip
	\begin{subfigure}[b]{0.7\textwidth}  
		\centering 
		\includegraphics[width=\textwidth]{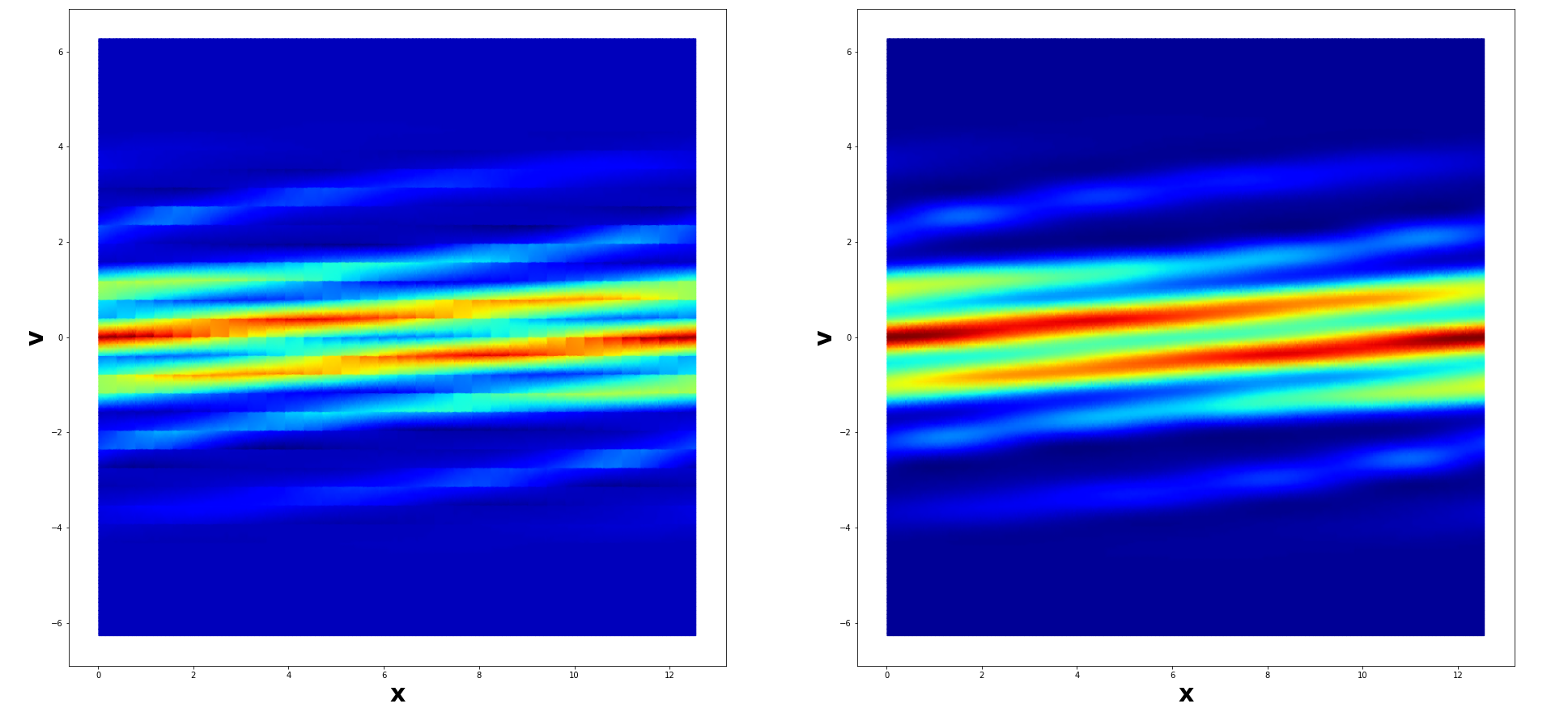}
		\caption{{\small $32 \times 32$}}    
		\label{fig:con_ld_k1_t10_32}
	\end{subfigure}
	\vskip\baselineskip
	\begin{subfigure}[b]{0.7\textwidth}  
		\centering 
		\includegraphics[width=\textwidth]{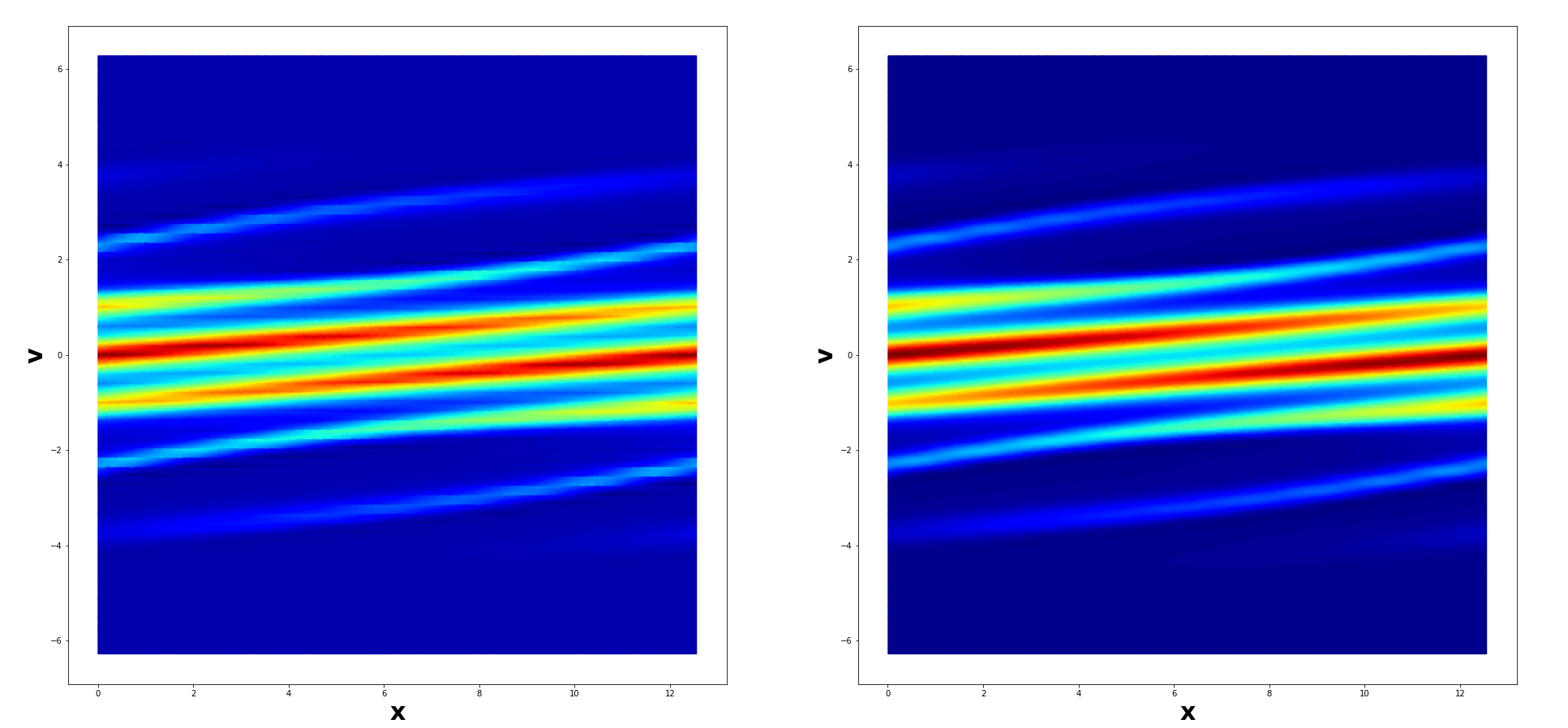}
		\caption{{\small $64 \times 64$}}    
		\label{fig:con_ld_k1_t10_64}
	\end{subfigure}
	\caption{\small Comparison of   contour plots before (left) and after post-processing (right) for different mesh-sizes. Landau damping, $k=1$ and $T=10.$} 
	\label{fig:con_ld_k1_t10} 
\end{figure}   
\begin{figure}[!htbp]
	\centering
	\begin{subfigure}[b]{0.7\textwidth}
		\centering
		\includegraphics[width=\textwidth]{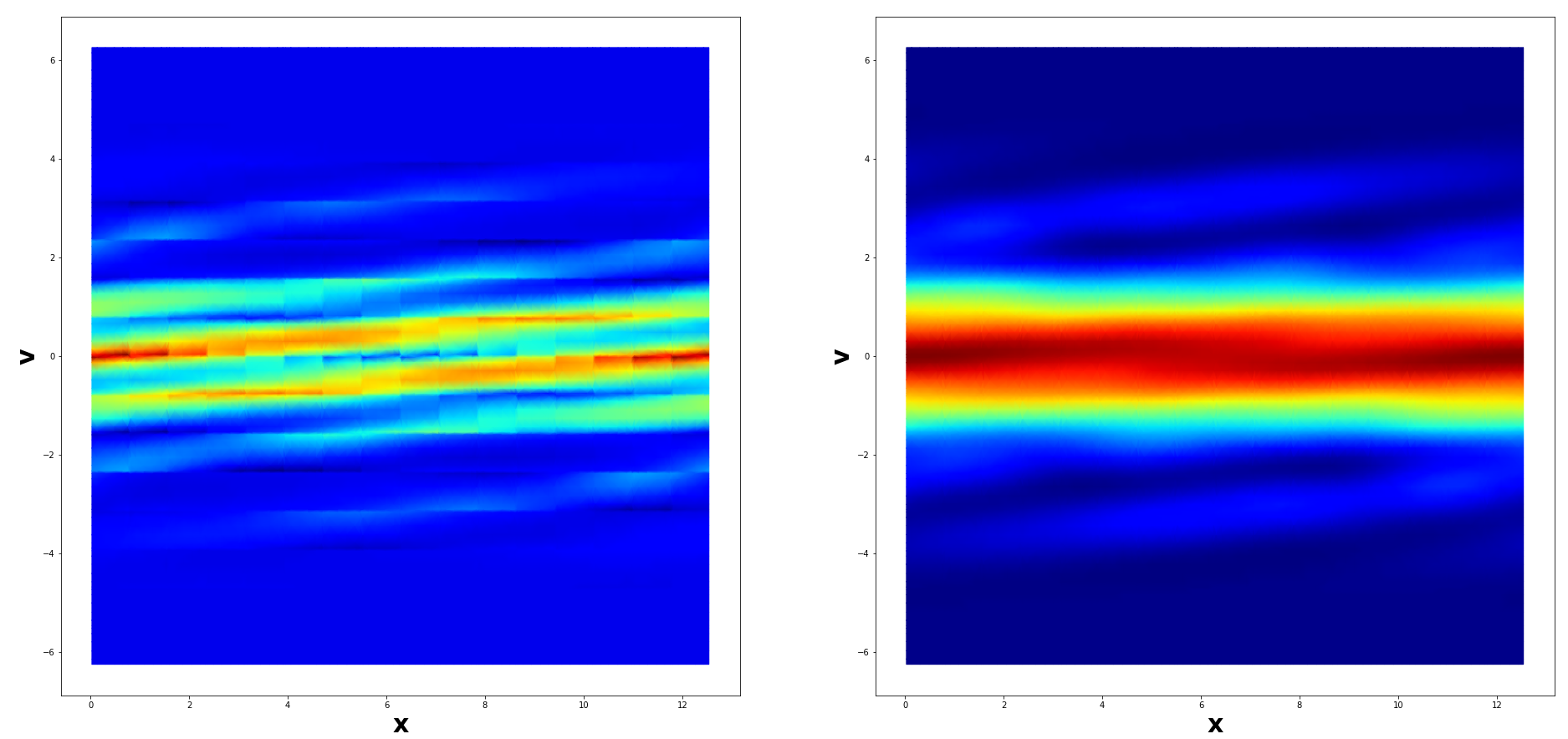}
		\caption{{\small $16\times 16$}}    
		\label{fig:con_ld_k2_t10_16}
	\end{subfigure}
	\vskip\baselineskip
	\begin{subfigure}[b]{0.7\textwidth}  
		\centering 
		\includegraphics[width=\textwidth]{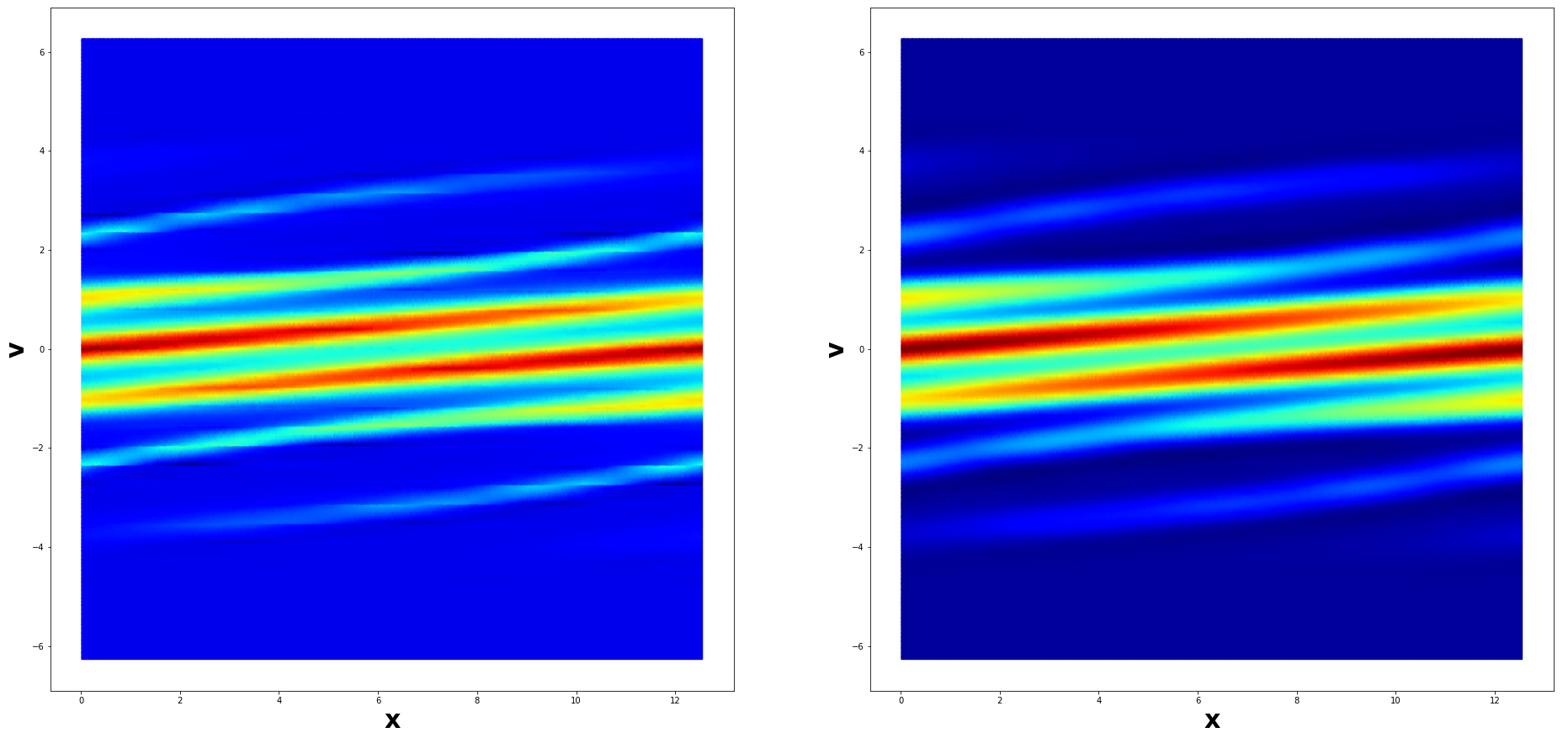}
		\caption{{\small $32 \times 32$}}    
		\label{fig:con_ld_k2_t10_32}
	\end{subfigure}
	\vskip\baselineskip
	\begin{subfigure}[b]{0.7\textwidth}  
		\centering 
		\includegraphics[width=\textwidth]{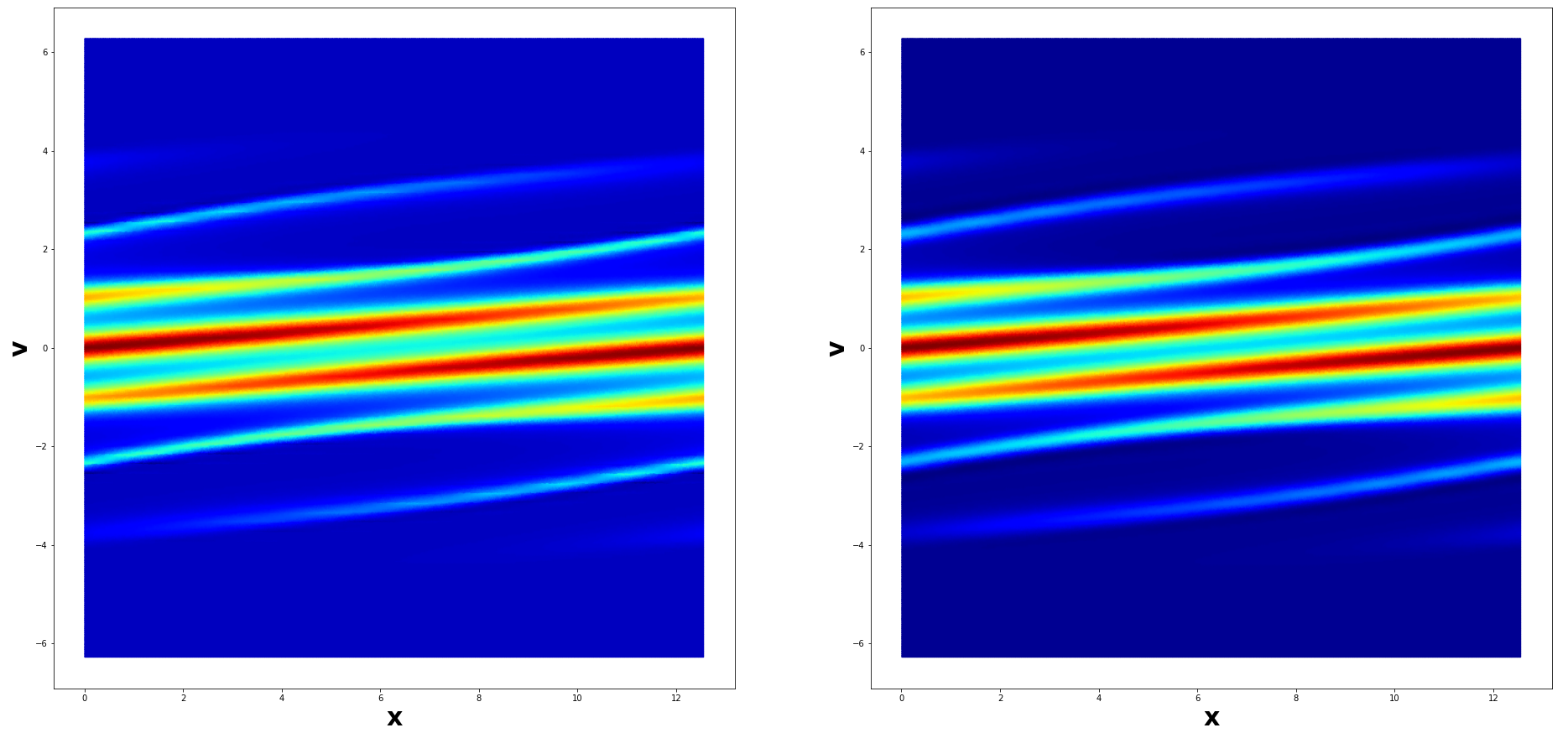}
		\caption{{\small $64 \times 64$}}    
		\label{fig:con_ld_k2_t10_64}
	\end{subfigure}
	\caption{\small Comparison of the definition of the contour plots before (left) and after post-processing (right) for different mesh-sizes. Landau damping, $k=2$ and $T=10.$} 
	\label{fig:con_ld_k2_t10} 
\end{figure}  

\begin{figure}[!htbp]
	\centering
	\begin{subfigure}[b]{0.7\textwidth}
		\centering
		\includegraphics[width=\textwidth]{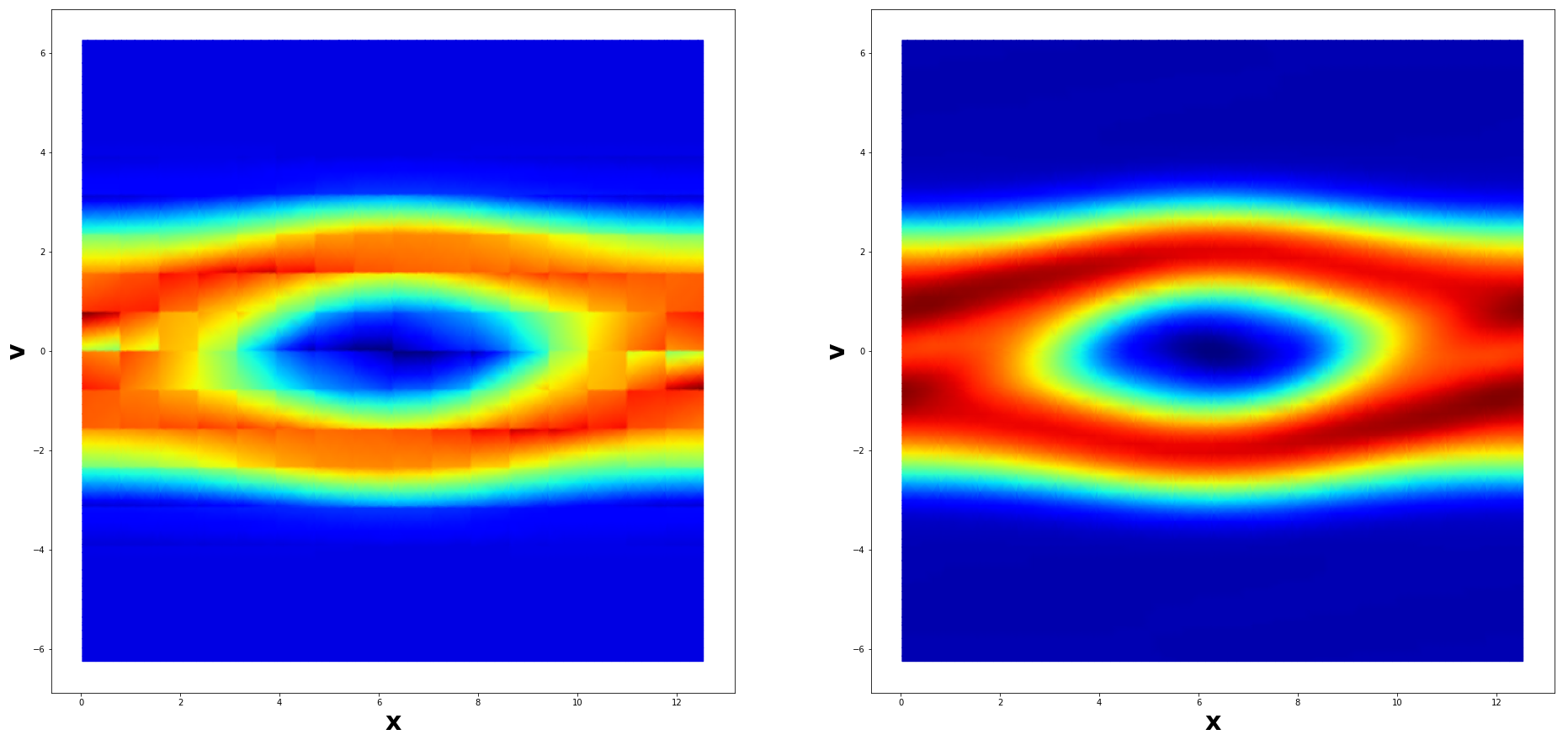}
		\caption{{\small $16\times 16$}}    
		\label{fig:con_ts_k1_t20_16}
	\end{subfigure}
	\vskip\baselineskip
	\begin{subfigure}[b]{0.7\textwidth}  
		\centering 
		\includegraphics[width=\textwidth]{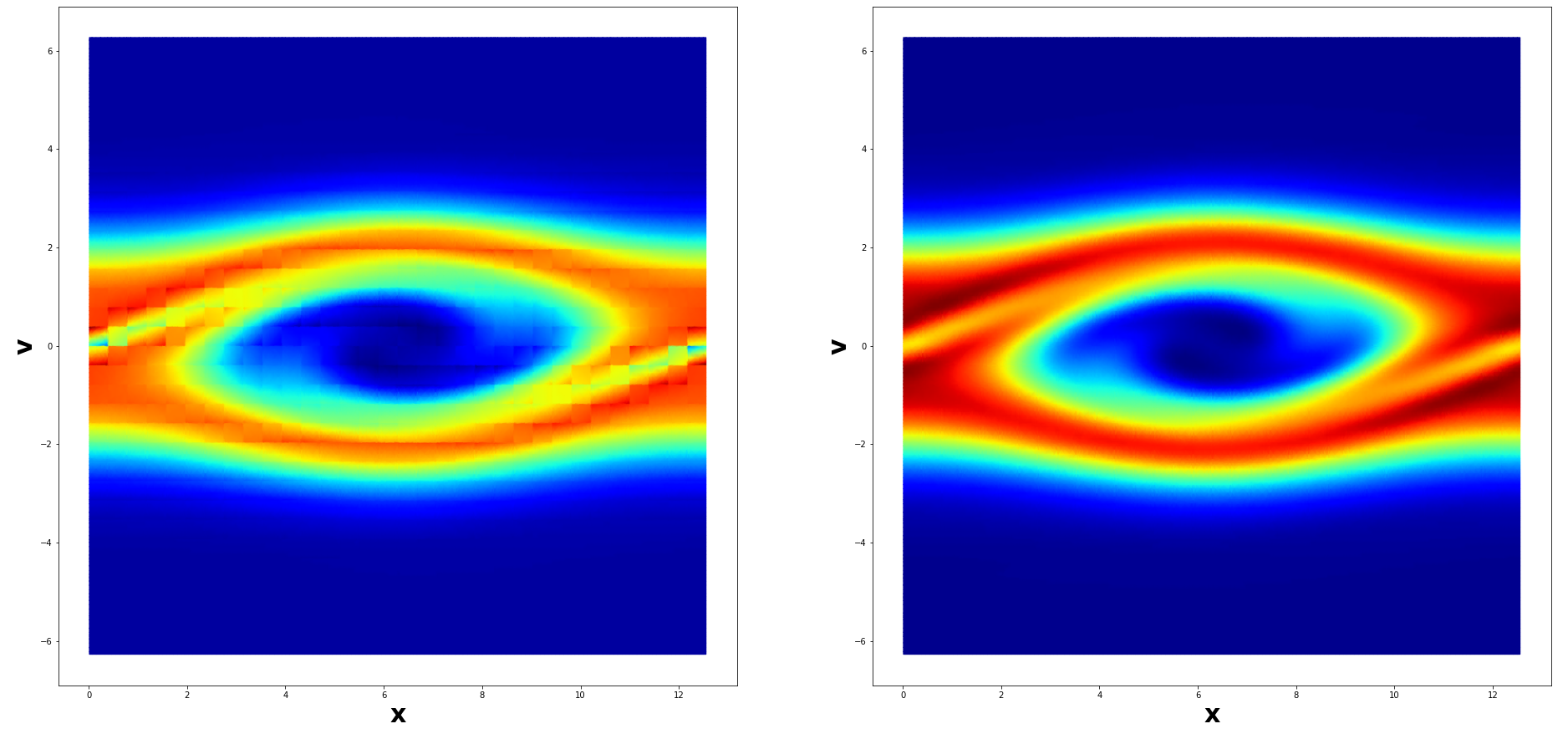}
		\caption{{\small $32\times 32$}}    
		\label{fig:con_ts_k1_t20_32}
	\end{subfigure}
	\vskip\baselineskip
	\begin{subfigure}[b]{0.7\textwidth}  
		\centering 
		\includegraphics[width=\textwidth]{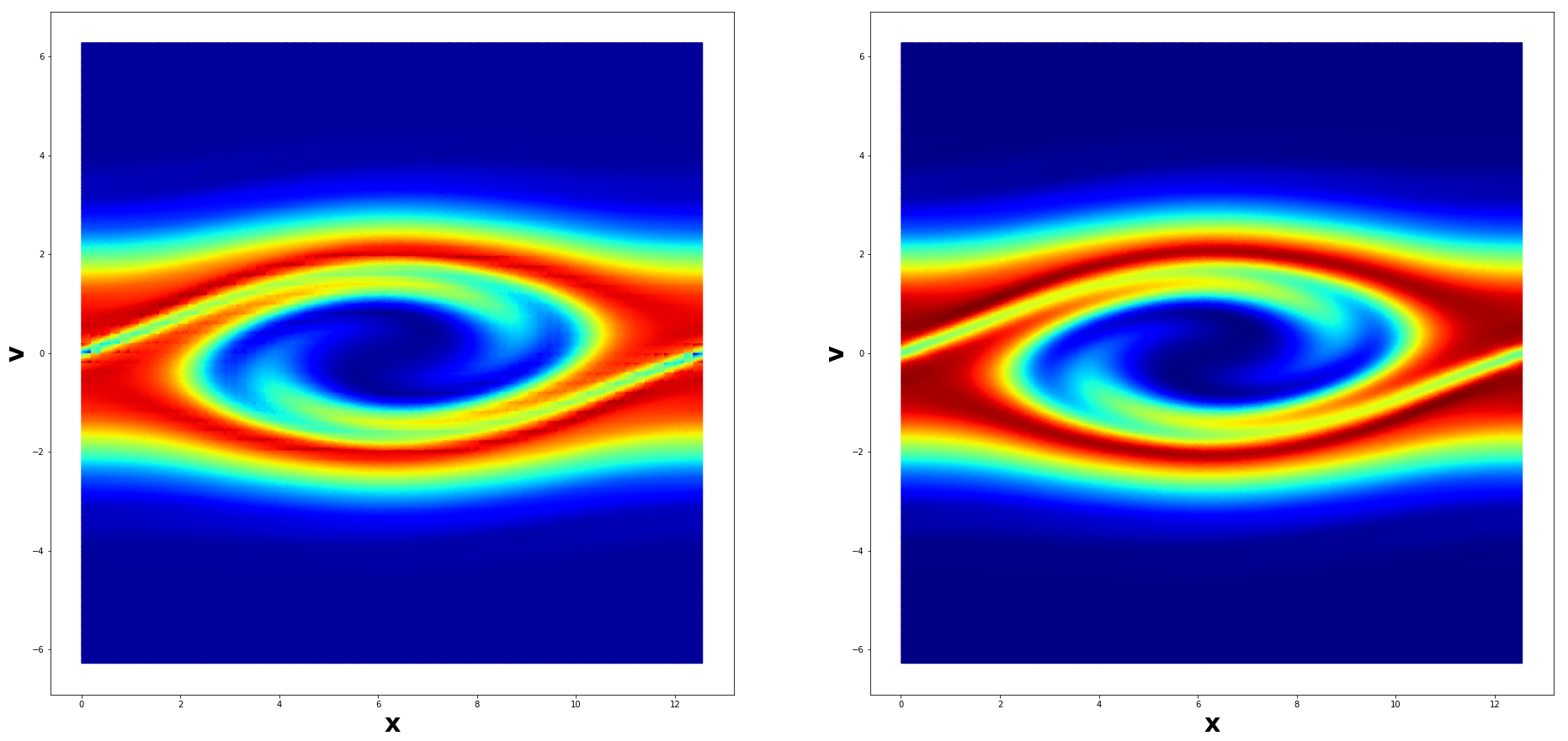}
		\caption{{\small $64 \times 64$}}    
		\label{fig:con_ts_k1_t20_64}
	\end{subfigure}
	\caption{\small Comparison of contour plots before (left) and after post-processing (right) for different mesh-sizes. Two stream instability, $k=1$ and $T=20.$} 
	\label{fig:con_ts_k1_t20} 
\end{figure}   

\begin{figure}[!htbp]
	\centering
	\begin{subfigure}[b]{0.7\textwidth}
		\centering
		\includegraphics[width=\textwidth]{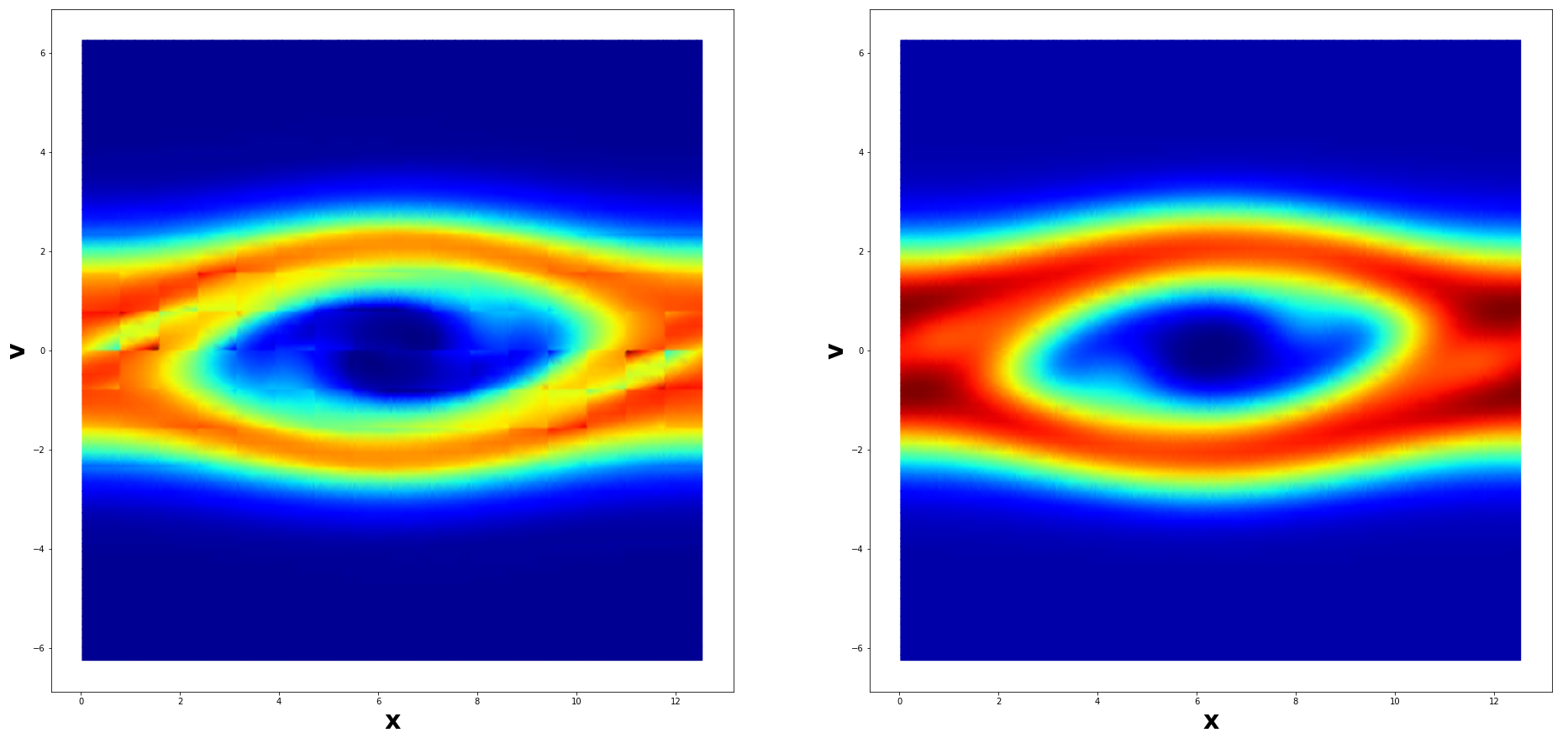}
		\caption{{\small $16\times 16$}}    
		\label{fig:con_ts_k2_t20_16}
	\end{subfigure}
	\vskip\baselineskip
	\begin{subfigure}[b]{0.7\textwidth}  
		\centering 
		\includegraphics[width=\textwidth]{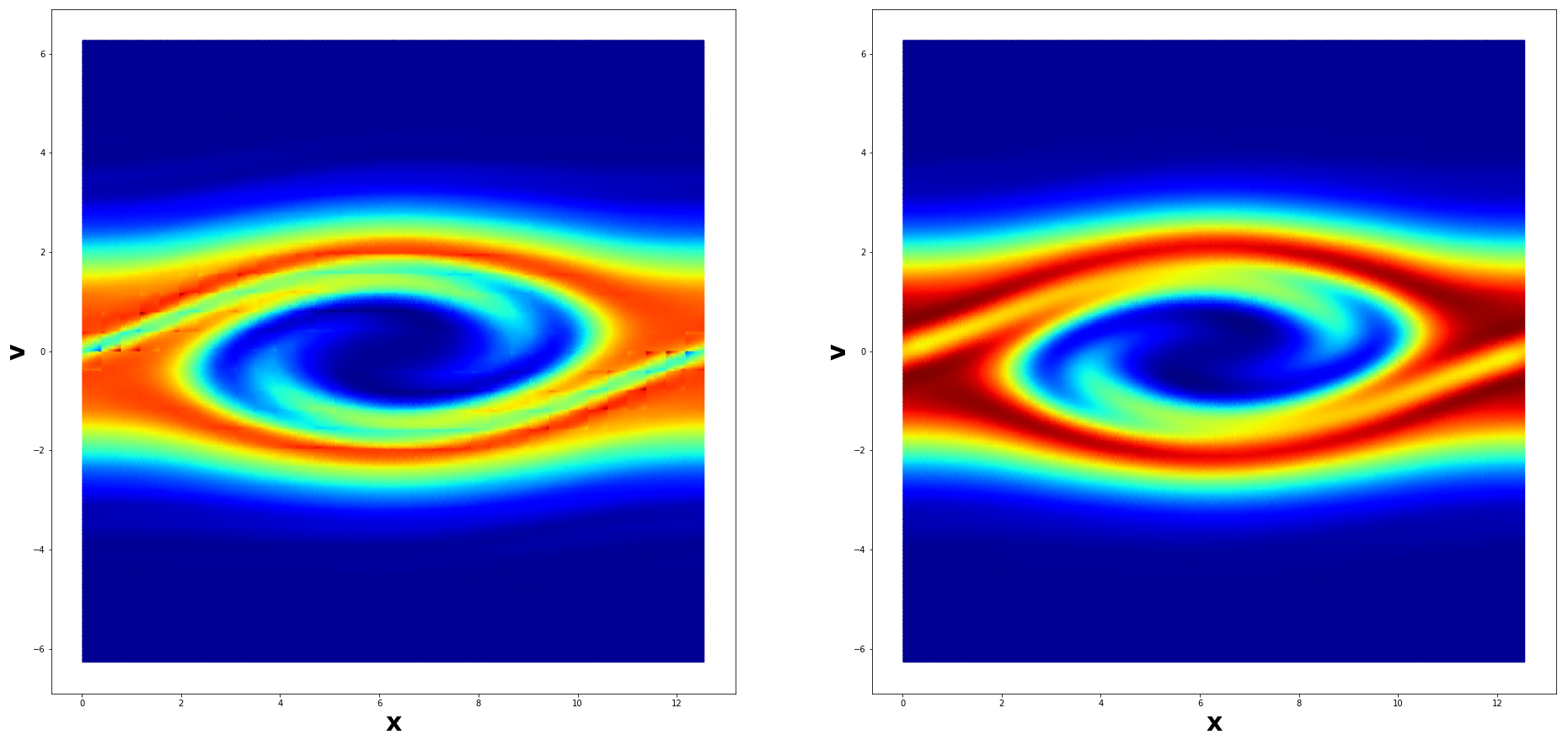}
		\caption{{\small $32\times 32$}}    
		\label{fig:con_ts_k2_t20_32}
	\end{subfigure}
	\vskip\baselineskip
	\begin{subfigure}[b]{0.7\textwidth}  
		\centering 
		\includegraphics[width=\textwidth]{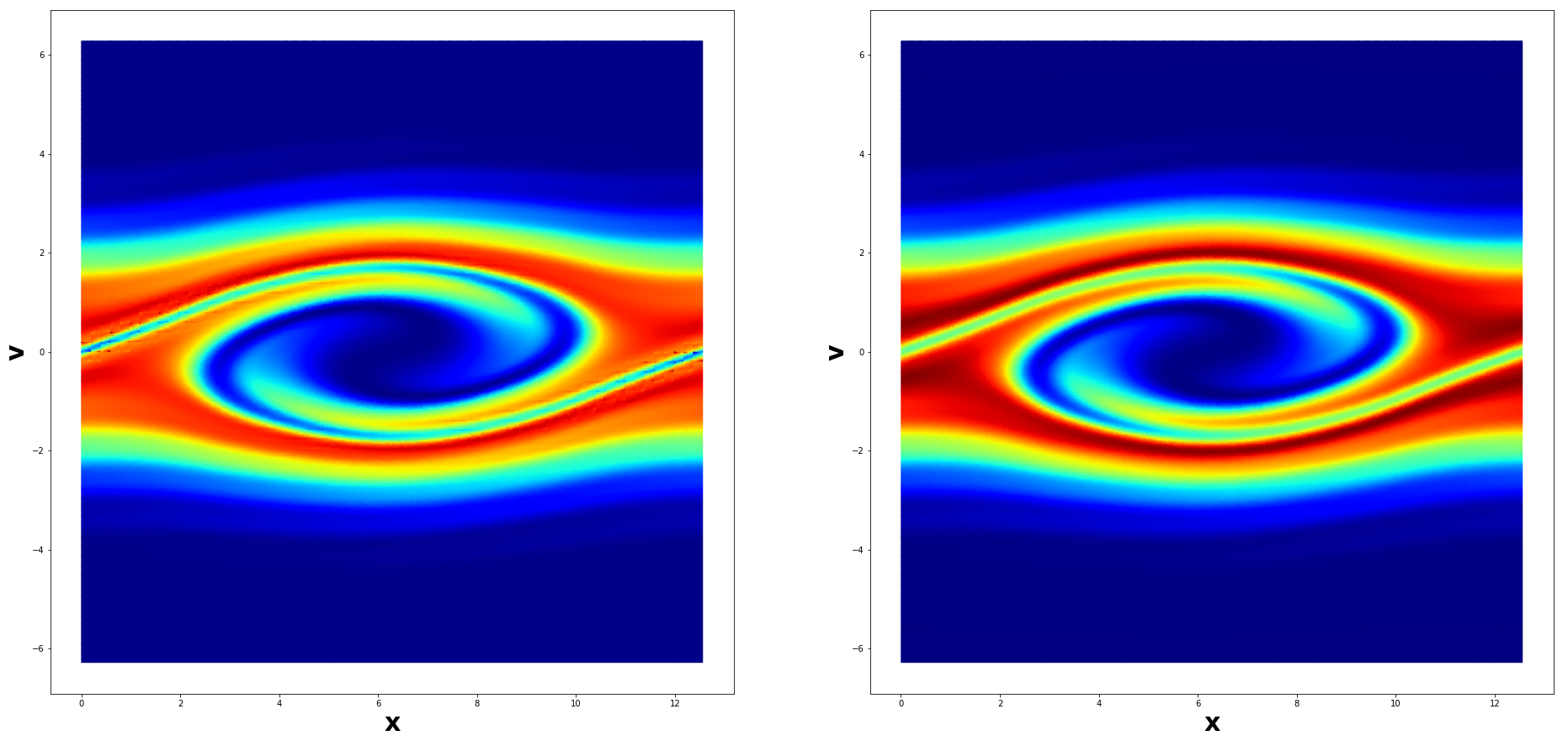}
		\caption{{\small $64 \times 64$}}    
		\label{fig:con_ts_k2_t20_64}
	\end{subfigure}
	\caption{\small Comparison of the definition of the contour plots before (left) and after post-processing (right) for different mesh-sizes. Two stream instability, $k=2$ and $T=20.$} 
	\label{fig:con_ts_k2_t20} 
\end{figure}   

\begin{figure}[!htbp]
	\centering
	\begin{subfigure}[b]{\textwidth}
		\centering
		\includegraphics[width=0.8\textwidth]{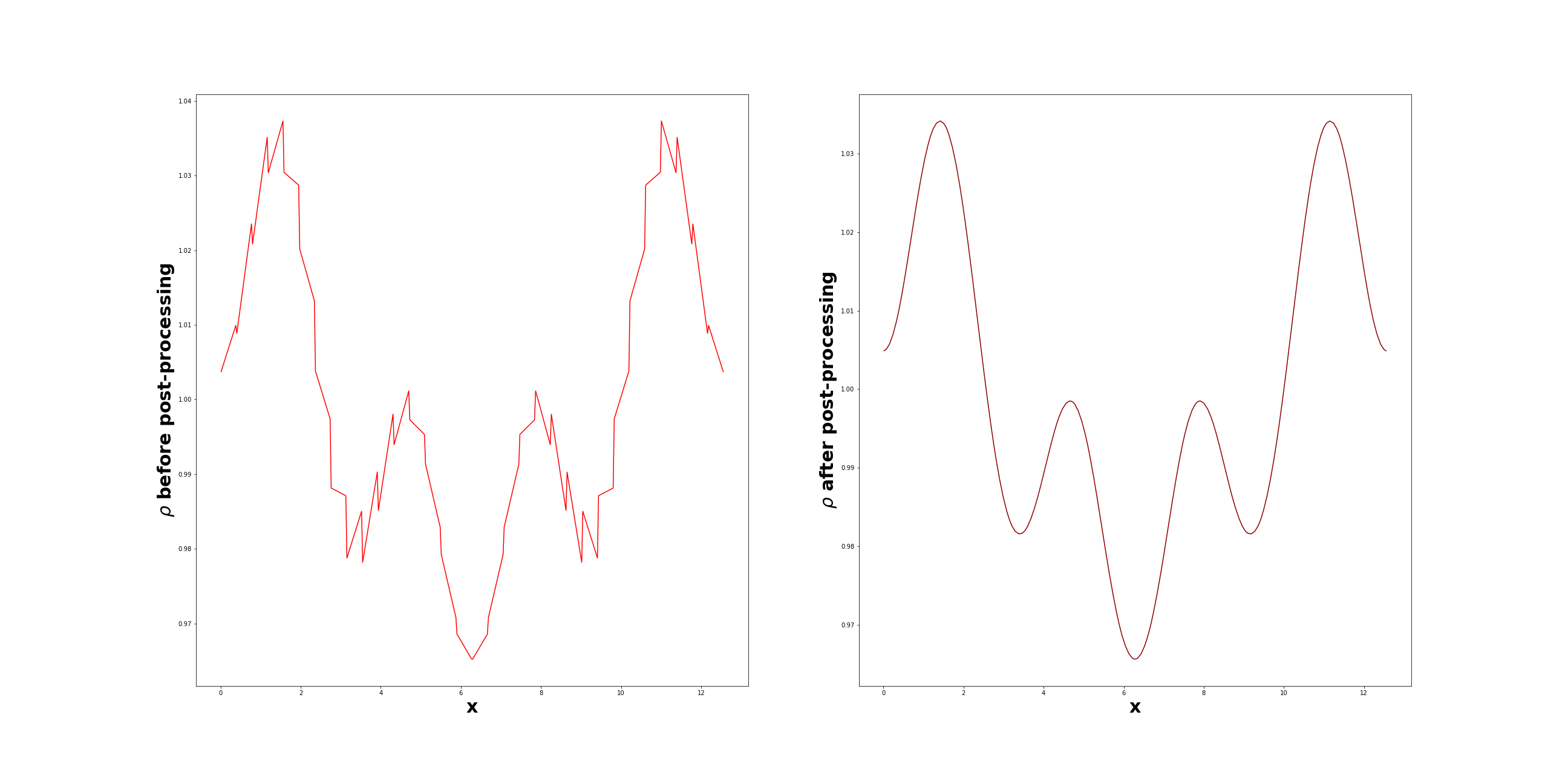}
		\caption{Particle density $\rho$.}
		\label{fig:particle_density}
	\end{subfigure}
	\vskip\baselineskip
	\begin{subfigure}[b]{\textwidth}  
		\centering
	\includegraphics[width=0.8\textwidth]{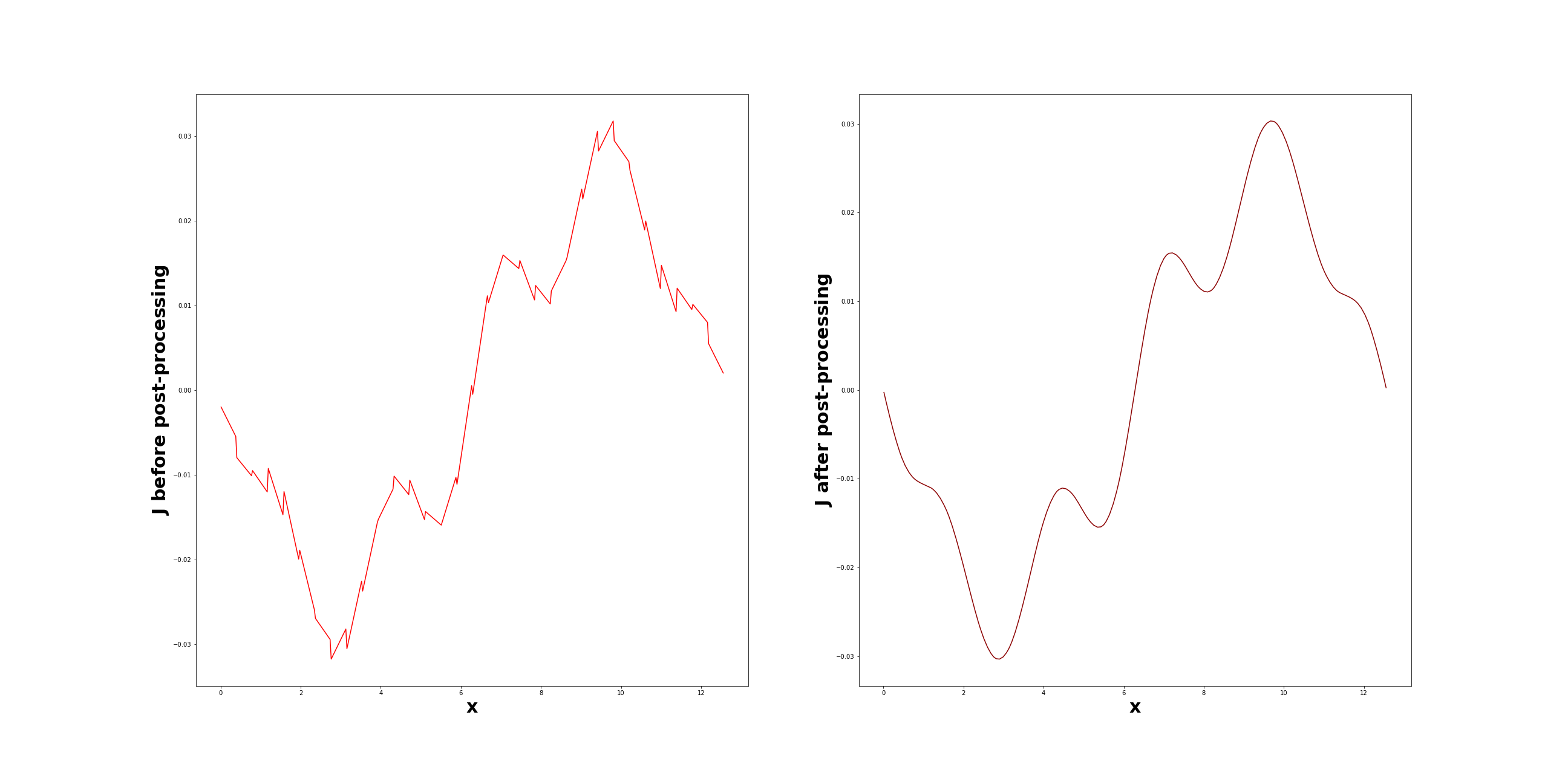}
	\caption{Current density $\mathbf{J}$.}
	\label{fig:current_density}
	\end{subfigure}
	\caption{\small Comparison of density plots for $\rho$ and $\mathbf{J}$ before (left) and after post-processing (right). Landau Damping $k=1$, mesh $32\times 32$. $T=10.$} 
	\label{fig:pp-density} 
\end{figure}  

%
\subsection{Vlasov-Maxwell example}

In this part, we will test our post-processor for the   VM system. Specifically we will use the streaming Weibel (SW) instability as an example. This is a reduced version of the VM equations with one spatial variable, $x_2$, and two velocity variables $v_1$ and $v_2.$ The  variables under consideration are the distribution function $f(x_2,v_1,v_2,t)$, a 2D electric field $\elec=(E_1(x_2,t),E_2(x_2,t),0)$ and a 1D magnetic field $\magn=(0,0,B_3(x_2,t))$ and the reduced VM system reads as
	\begin{subequations}\label{eq:wi_system}
	\begin{gather}
		\partial_t f+v_2 f_{x_2}+(E_1+v_2B_3)f_{v_1}+(E_2-v_1 B_3)f_{v_2}=0,\label{eq:wi_main_eq}\\
		\frac{\partial B_3}{\partial t}=\frac{\partial E_1}{\partial x_2},\quad\frac{\partial E_1}{\partial t}=\frac{\partial B_3}{\partial x_2}-j_1,\quad \frac{\partial E_2}{\partial t}=-j_2,\label{eq:wi_elec_magn}
	\end{gather}
\end{subequations}	
where 
\begin{equation}
	j_1=\int_{-\infty}^{\infty}\int_{-\infty}^{\infty}f(x_2,v_1,v_2,t)v_1\,dv_1 dv_2,\quad j_2=\int_{-\infty}^{\infty}\int_{-\infty}^{\infty}f(x_2,v_1,v_2,t)v_2\,dv_1 dv_2.
\end{equation}	
The initial conditions are given by
	\begin{subequations}\label{eq:wi_icond}
	\begin{gather}
	    f(x_2,v_1,v_2,0)=\frac{1}{\pi\beta}e^{-v_2^2/\beta}[\delta e^{-(v_1-\omega_{0,1})^2/\beta}+(1-\delta)e^{-(v_1+\omega_{0,2})^2/\beta}],\label{eq:wi_f_init}\\
		E_1(x_2,v_1,v_2,0)=E_2(x_2,v_1,v_2,0)=0,\quad B_3(x_2,v_1,v_2,0)=b\sin(\kappa_0 x_2),\label{eq:wi_elec_,magn_icond}
	\end{gather}
\end{subequations}
which for $b=0$ is an equilibrium state composed of counter-streaming beams propagating perpendicular to the direction of inhomogeneity. Following \cite{califano1998ksw,cheng2014discontinuous}, we trigger the instability by taking $\beta=0.01$ , $b=0.001$ (the amplitude of the initial perturbation of the magnetic field). Here, $\Omega_x=[0,L_y]$ , where $L_y=2\pi/\kappa_0$, and we set $\Omega_{v}=[-1.8,1.8]^2$. We consider the following set of parameters, 
\begin{equation*}
\delta=0.5,\,\omega_{0,1}=\omega_{0,2}=0.3,\,\kappa_0=0.2. 
\end{equation*}

In Table \ref{tab:maxw_eq_error}, we run the VM system with initial condition from SW instability to $T=5$ and then back to $T=0$, we then apply the SIAC filter and compare it with the initial conditions.  We use a third order TVD-RK method  as the time integrator.  To make sure the spatial error dominates,  we take $\Delta t=O(\Delta x)$ for $\mathbb{P}^1$ and $\Delta t=O(\Delta x^{5/3})$ for $\mathbb{P}^2$, in both cases we used $\mathrm{CFL}=0.1$.   From the table we can observe $(k+1)$-th order of convergence for the DG solution before post-processing for $f$, $E_1$, $E_2$ and $B_3$. After post-processing we can see  overall the order of convergence improves to $O(h^{2k+1/2})$.        
	  	
\begin{table}[!htbp]
	\centering
	\begin{tabular}{|ccccccccc|}
		\hline
		\multicolumn{9}{|c|}{Before post-processing}                                                                                                                                                                                                                                                                                                                           \\ \hline
		\multicolumn{1}{|c|}{Mesh}                                                                 & \multicolumn{1}{c|}{Error $f$}   & \multicolumn{1}{c|}{Order} & \multicolumn{1}{c|}{Error $\mathbf{B}_3$}   & \multicolumn{1}{c|}{Order} & \multicolumn{1}{c|}{Error $\mathbf{E}_1$}   & \multicolumn{1}{c|}{Order} & \multicolumn{1}{c|}{Error $\mathbf{E}_2$}   & Order \\ \hline
		\multicolumn{9}{|c|}{$\mathbb{P}^1$}                                                                                                                                                                                                                                                                                                                                   \\ \hline
		\multicolumn{1}{|c|}{$20\times20\times20$}                                                 & \multicolumn{1}{c|}{2.20E-01}    & \multicolumn{1}{c|}{-}     & \multicolumn{1}{c|}{2.61E-06}               & \multicolumn{1}{c|}{-}     & \multicolumn{1}{c|}{2.12E-06}               & \multicolumn{1}{c|}{-}     & \multicolumn{1}{c|}{5.31E-06}               & -     \\
		\multicolumn{1}{|c|}{$40\times40\times40$}                                                 & \multicolumn{1}{c|}{7.17E-02}    & \multicolumn{1}{c|}{1.61}  & \multicolumn{1}{c|}{6.54E-07}               & \multicolumn{1}{c|}{2.00}  & \multicolumn{1}{c|}{7.06E-07}               & \multicolumn{1}{c|}{1.58}  & \multicolumn{1}{c|}{5.46E-07}               & 3.28  \\
		\multicolumn{1}{|c|}{\begin{tabular}[c]{@{}c@{}}$80\times 80\times 80$\end{tabular}}     & \multicolumn{1}{c|}{1.92E-02}    & \multicolumn{1}{c|}{1.90}  & \multicolumn{1}{c|}{1.63E-07}               & \multicolumn{1}{c|}{2.00}  & \multicolumn{1}{c|}{1.96E-07}               & \multicolumn{1}{c|}{1.85}  & \multicolumn{1}{c|}{7.05E-08}               & 2.95  \\
		\multicolumn{1}{|c|}{\begin{tabular}[c]{@{}c@{}}$160\times 160\times 160$\end{tabular}} & \multicolumn{1}{c|}{4.89E-03}    & \multicolumn{1}{c|}{1.98}  & \multicolumn{1}{c|}{4.07E-08}               & \multicolumn{1}{c|}{2.00}  & \multicolumn{1}{c|}{5.13E-08}               & \multicolumn{1}{c|}{1.94}  & \multicolumn{1}{c|}{6.40E-09}               & 3.46  \\ \hline
		\multicolumn{9}{|c|}{$\mathbb{P}^2$}                                                                                                                                                                                                                                                                                                                                   \\ \hline
		\multicolumn{1}{|c|}{$20\times20\times20$}                                                 & \multicolumn{1}{c|}{1.07E-01}    & \multicolumn{1}{c|}{-}     & \multicolumn{1}{c|}{2.56E-07}               & \multicolumn{1}{c|}{-}     & \multicolumn{1}{c|}{2.49E-07}               & \multicolumn{1}{c|}{-}     & \multicolumn{1}{c|}{1.02E-06}               &       \\
		\multicolumn{1}{|c|}{\begin{tabular}[c]{@{}c@{}}$40 \times 40 \times 40$\end{tabular}}   & \multicolumn{1}{c|}{1.64E-02}    & \multicolumn{1}{c|}{2.70}  & \multicolumn{1}{c|}{3.14E-08}               & \multicolumn{1}{c|}{3.03}  & \multicolumn{1}{c|}{2.93E-08}               & \multicolumn{1}{c|}{3.09}  & \multicolumn{1}{c|}{9.72E-08}               & 3.40  \\
		\multicolumn{1}{|c|}{\begin{tabular}[c]{@{}c@{}}$80\times 80\times 80$\end{tabular}}     & \multicolumn{1}{c|}{2.23E-03}    & \multicolumn{1}{c|}{2.88}  & \multicolumn{1}{c|}{1.63E-09}               & \multicolumn{1}{c|}{4.27}  & \multicolumn{1}{c|}{1.90E-09}               & \multicolumn{1}{c|}{3.95}  & \multicolumn{1}{c|}{6.93E-09}               & 3.81  \\
		\multicolumn{1}{|c|}{\begin{tabular}[c]{@{}c@{}}$160\times 160\times 160$\end{tabular}} & \multicolumn{1}{c|}{2.92E-04}    & \multicolumn{1}{c|}{2.93}  & \multicolumn{1}{c|}{1.41E-10}               & \multicolumn{1}{c|}{3.52}  & \multicolumn{1}{c|}{1.72E-10}               & \multicolumn{1}{c|}{3.46}  & \multicolumn{1}{c|}{2.46E-10}               & 4.81  \\ \hline
		\multicolumn{9}{|c|}{After post-processing}                                                                                                                                                                                                                                                                                                                            \\ \hline
		\multicolumn{1}{|c|}{Mesh}                                                                 & \multicolumn{1}{c|}{Error $f^*$} & \multicolumn{1}{c|}{Order} & \multicolumn{1}{c|}{Error $\mathbf{B}^*_3$} & \multicolumn{1}{c|}{Order} & \multicolumn{1}{c|}{Error $\mathbf{E}^*_1$} & \multicolumn{1}{c|}{Order} & \multicolumn{1}{c|}{Error $\mathbf{E}^*_2$} & Order \\ \hline
		\multicolumn{9}{|c|}{$\mathbb{P}^1$}                                                                                                                                                                                                                                                                                                                                   \\ \hline
		\multicolumn{1}{|c|}{$20\times20\times20$}                                                 & \multicolumn{1}{c|}{2.95E-01}    & \multicolumn{1}{c|}{-}     & \multicolumn{1}{c|}{3.17E-07}               & \multicolumn{1}{c|}{-}     & \multicolumn{1}{c|}{1.08E-07}               & \multicolumn{1}{c|}{-}     & \multicolumn{1}{c|}{5.08E-06}               & -     \\
		\multicolumn{1}{|c|}{\begin{tabular}[c]{@{}c@{}}$40\times 40\times 40$\end{tabular}}     & \multicolumn{1}{c|}{6.13E-02}    & \multicolumn{1}{c|}{2.27}  & \multicolumn{1}{c|}{7.16E-08}               & \multicolumn{1}{c|}{2.14}  & \multicolumn{1}{c|}{1.49E-08}               & \multicolumn{1}{c|}{2.87}  & \multicolumn{1}{c|}{4.38E-07}               & 3.54  \\
		\multicolumn{1}{|c|}{\begin{tabular}[c]{@{}c@{}}$80\times 80\times 80$\end{tabular}}     & \multicolumn{1}{c|}{5.87E-03}    & \multicolumn{1}{c|}{3.38}  & \multicolumn{1}{c|}{1.12E-08}               & \multicolumn{1}{c|}{2.68}  & \multicolumn{1}{c|}{3.11E-09}               & \multicolumn{1}{c|}{2.26}  & \multicolumn{1}{c|}{6.33E-08}               & 2.79  \\
		\multicolumn{1}{|c|}{\begin{tabular}[c]{@{}c@{}}$160\times 160\times 160$\end{tabular}} & \multicolumn{1}{c|}{4.19E-04}    & \multicolumn{1}{c|}{3.81}  & \multicolumn{1}{c|}{2.01E-09}               & \multicolumn{1}{c|}{2.48}  & \multicolumn{1}{c|}{7.47E-10}               & \multicolumn{1}{c|}{2.06}  & \multicolumn{1}{c|}{6.22E-09}               & 3.35  \\ \hline
		\multicolumn{9}{|c|}{$\mathbb{P}^2$}                                                                                                                                                                                                                                                                                                                                   \\ \hline
		\multicolumn{1}{|c|}{$20\times20\times20$}                                                 & \multicolumn{1}{c|}{2.89E-01}    & \multicolumn{1}{c|}{-}     & \multicolumn{1}{c|}{1.24E-08}               & \multicolumn{1}{c|}{-}     & \multicolumn{1}{c|}{9.06E-09}               & \multicolumn{1}{c|}{-}     & \multicolumn{1}{c|}{4.41E-07}               & -     \\
		\multicolumn{1}{|c|}{\begin{tabular}[c]{@{}c@{}}$40 \times 40 \times 40$\end{tabular}}   & \multicolumn{1}{c|}{4.58E-02}    & \multicolumn{1}{c|}{2.66}  & \multicolumn{1}{c|}{5.61E-10}               & \multicolumn{1}{c|}{4.46}  & \multicolumn{1}{c|}{2.97E-10}               & \multicolumn{1}{c|}{4.93}  & \multicolumn{1}{c|}{2.63E-08}               & 4.07  \\
		\multicolumn{1}{|c|}{\begin{tabular}[c]{@{}c@{}}$80\times 80\times 80$\end{tabular}}     & \multicolumn{1}{c|}{2.03E-03}    & \multicolumn{1}{c|}{4.49}  & \multicolumn{1}{c|}{2.94E-11}               & \multicolumn{1}{c|}{4.25}  & \multicolumn{1}{c|}{1.31E-11}               & \multicolumn{1}{c|}{4.50}  & \multicolumn{1}{c|}{2.57E-09}               & 3.36  \\
		\multicolumn{1}{|c|}{\begin{tabular}[c]{@{}c@{}}$160\times 160\times 160$\end{tabular}} & \multicolumn{1}{c|}{4.43E-05}    & \multicolumn{1}{c|}{5.52}  & \multicolumn{1}{c|}{1.65E-12}               & \multicolumn{1}{c|}{4.15}  & \multicolumn{1}{c|}{5.55E-13}               & \multicolumn{1}{c|}{4.56}  & \multicolumn{1}{c|}{1.12E-10}               & 4.53  \\ \hline
	\end{tabular}

\caption{$L^2$ errors for the numerical solution (Above) and the post-processed solution (Below). SW instability. }
\label{tab:maxw_eq_error}
\end{table}

In Figure \ref{fig:wi_error_comparision} we plot a cross-section of the errors of the numerical solution at $x_2\approx 0.15\pi$ before and after post-processing for $\mathbb{P}^1$ using $80\times 80\times 80$ elements. We can see that before post-processing that the errors are highly oscillatory, and after post-processing the error surface is smooth out and the error is much smaller in magnitude.  In Figure \ref{fig:fields_error_wi} we plot the errors of $E_1,\,E_2$ and $B_3$, we used the same number of elements as in Figure \ref{fig:wi_error_comparision}, We can clearly see similar conclusions.     

\begin{figure}[!htbp]
	\centering
	\includegraphics[width=\textwidth]{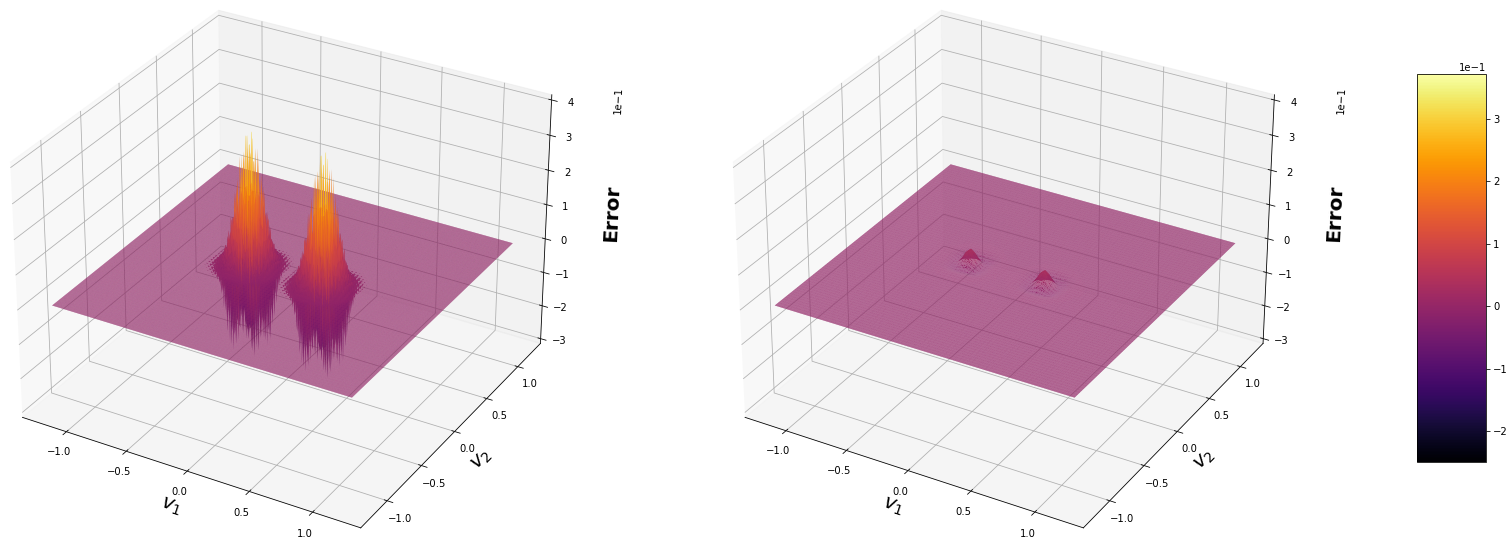}
	\caption{Cross-sectional plot of the error for $f$ at $x_2\approx 0.15\pi$,  before (on the left) and after post-processing (on the right) for $80^3$ elements and $\mathbb{P}^1.$ SW instability.}	\label{fig:wi_error_comparision}
\end{figure} 
\begin{figure}[!htbp]
	\centering
	\begin{subfigure}[b]{0.6\textwidth}
		\centering
		\includegraphics[width=\textwidth]{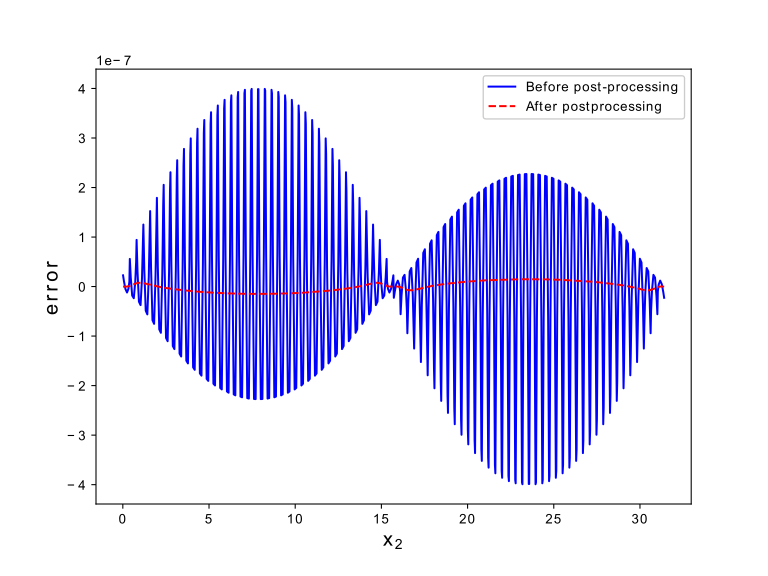}
		\caption{{\small Error for $\magn_3$}.}  
	\end{subfigure}
	\begin{subfigure}[b]{0.6\textwidth}
		\centering
		\includegraphics[width=\textwidth]{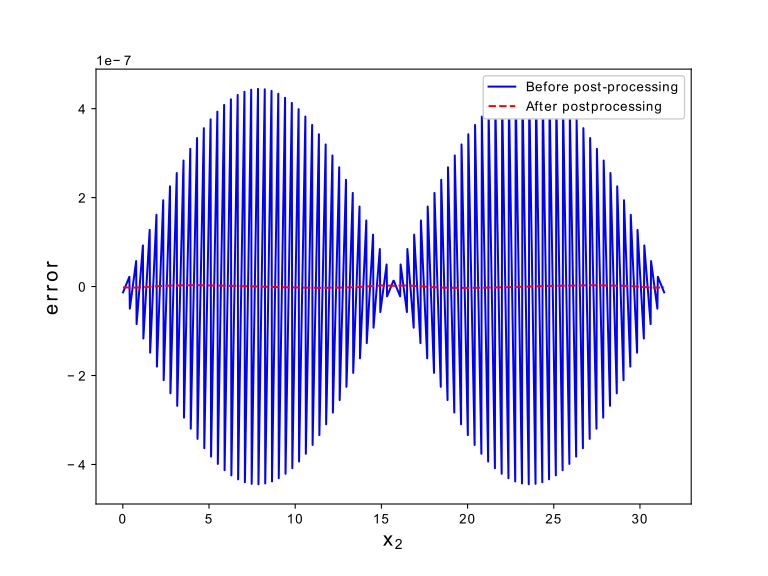}
		\caption{{\small Error for $\elec_1$}.}  
	\end{subfigure}
	\begin{subfigure}[b]{0.6\textwidth}
	\centering
	\includegraphics[width=\textwidth]{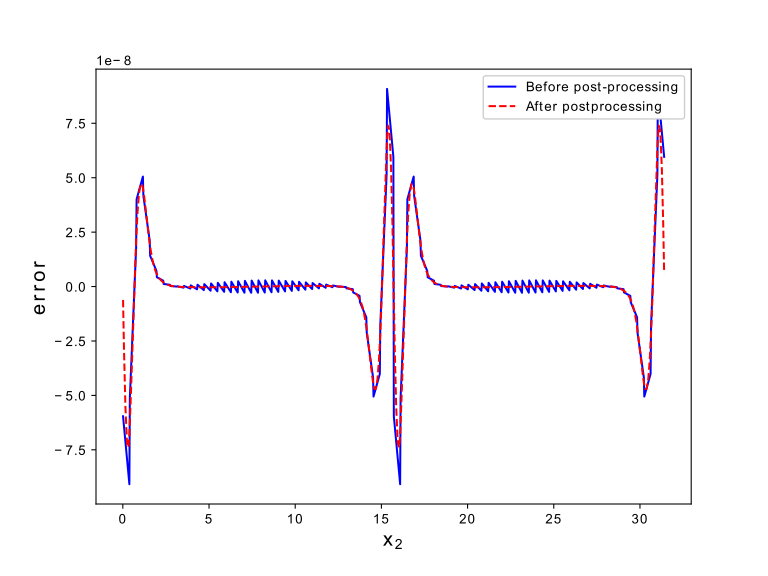}
	\caption{{\small Error for $\elec_2$}.}  
\end{subfigure}
	\caption{\small Errors before (solid line) and after post-processing (dashed line) for the different fields using mesh size of $80\times 80\times 80$ and $\mathbb{P}^1$. $T=10$. SW instability.} 
	\label{fig:fields_error_wi} 
\end{figure}

\section{Concluding Remarks}
\label{sec:conclude}

In this paper, we proved theoretically and demonstrated computationally the effectiveness of the SIAC filter to the DG solutions of the nonlinear VM system. We proved the superconvergence of order $(2k+\frac{1}{2})$ in the negative norm of the DG solutions. This is nontrivial for nonlinear systems, and is achieved by identifying a suitable dual problem. The numerical experiments verify the performance of the filter in reducing spurious oscillations in the numerical errors. For low order $k$, the resolution of the numerical solution is greatly enhanced, which is highly desirable for long time kinetic simulations.
In the future, we plan to prove superconvergence for the divided difference of the numerical solution to fully justify the enhanced resolution of the post-processed solution.

\appendix
\begin{appendices}

\section{Proof of Lemma \ref{lem:regularity_est}}
 By using equation \eqref{eq:gen_du_phi}, the divergence free properties of $\mathbf{A_1}, \mathbf{A_2}$ and the boundary conditions, we have the following 
 
\begin{align*}
	\frac{1}{2}\frac{d}{dt}\norma{\varphi}^2&  =-\int_{\Omega}(\mathbf{A_3}\cdot\F)\varphi\,d\vecx d\vecv
								 \leq   C(\norma{\varphi}^2+\norma{\F}^2),
\end{align*}
where $C$ depends on $\norma{\mathbf{A_3}}_{L^\infty((0,T); L^{\infty}(\Omega))}.$
 On the other hand using equations \eqref{eq:gen_du_F} and \eqref{eq:gen_du_D}, Gauss theorem on the physical space integrals and integration by parts on the velocity space variables,   
\begin{align*}
	\frac{1}{2}\frac{d}{dt}\norma{\F}^2+\frac{1}{2}\frac{d}{dt}\norma{\D}^2&=\int_{\Omega_x}(\rot\D\cdot \F-\rot\F\cdot \D)\,d\vecx-\int_{\Omega_v}\varphi \nabla_{\vecv}g \cdot \F\,d\vecx d\vecv+\int_{\Omega_v}\varphi(\vecv\times\nabla_{\vecv}g)\D\,d\vecx d\vecv\\
																		   &=-\int_{\Omega}\varphi \nabla_{\vecv}g \cdot \F\,d\vecx d\vecv+\int_{\Omega}\varphi(\vecv\times\nabla_{\vecv}g)\D\,d\vecx d\vecv\\																		   &\leq C\left(\norma{\F}^2+\norma{\D}^2+\norma{\varphi}^2\right),
\end{align*}
where $C$ depends on $\norma{{g}}_{L^\infty((0,T); W^{1,\infty}(\Omega))}.$

Now we add the tow inequalities above, to obtain
\begin{equation}
	\frac{1}{2}\frac{d}{dt}\norma{\varphi}^2+\frac{1}{2}\frac{d}{dt}\norma{\F}^2+\frac{1}{2}\frac{d}{dt}\norma{\D}^2\leq C\left(\norma{\F}^2+\norma{\D}^2+\norma{\varphi}^2\right), 
\end{equation}	
where $C$ depends on $\norma{\mathbf{A_3}}_{L^\infty((0,T); L^{\infty}(\Omega))}$ and $\norma{{g}}_{L^\infty((0,T); W^{1,\infty}(\Omega))}.$
An application of Gronwall's inequality allow us to conclude. Now since we are considering the full Sobolev norm, we still need to estimate the $L^2$ norms of the higher order derivatives $\partial_{\vecx}^{\beta}\partial_{\vecv}^{\gamma}$, to do so we apply $\partial_{\vecx}^{\beta}\partial_{\vecv}^{\gamma}$ to the system \eqref{eq:reg_equation} and then we repeat the same steps that we took above.
  
\section{Proof of Lemma \ref{lem:proj_estimate} }
		By the definition of $\Pi^k$, 
		\begin{align*}
			(f_0-\Pi^kf_0,\varphi(0))_{\Omega}&=(f_0-\Pi^kf_0,\varphi(0)-\Pi^k\varphi(0))_{\Omega}\\
			&\leq \norma{f_0-\Pi^kf_0}\norma{\varphi(0)-\Pi^k\varphi(0)}\\
			&\leq Ch^{k+1}\norma{f_0}_{k+1,\Omega}h^{k+1}\norma{\varphi(0)}_{k+1,\Omega}.
		\end{align*}
	The last line was an application of the first part of Lemma \ref{lem:approximation_lemma}. By the same lines we obtain analogous results for the $\elec$ and $\magn$ parts. The conclusion follows by grouping them all together and an application of Cauchy-Schwarz inequality. 
	
\section{Proof of Lemma \ref{lem:residual} }

		Due to the definition of the projection operators, $((f_h)_t,\varphi-\chi)_{\Omega}=0$, $((\elec_h)_t,\F-\xi)_{\Omega_x}=0,$ and $((\magn_h)_t,\D-\eta)_{\Omega_x}=0,$ and
$
			l_h(\curre_h;\F-\xi)=-(\curre_h,\F-\xi)_{\Omega_x}=0,
$
we have
		\begin{equation*}
			\Theta_{\rm N}=\int_{0}^{T}-a_h(f_h,\elec_h,\magn_h;\zeta_h^{\varphi})-b_h(\elec_h,\magn_h;\zeta_h^{\F},\zeta_h^{\D})\,d\tau.
		\end{equation*}
From its definition,
		\begin{align*}
			b_h(\elec_h,\magn_h;\zeta_h^{\F},\zeta_h^{\D})&=\int_{\Thx}\elec_h\cdot\rot\zeta_h^{\D}\,d\vecx-\int_{\Thx}\magn_h\cdot\rot\zeta_h^{\F}\,d\vecx\\
			&+\int_{\edgesx}\widetilde{\elec_h}\cdot[\zeta_h^{\D}]_{\tau}\,ds_{\vecx}-\int_{\edgesx}\widetilde{\magn_h}\cdot[\zeta_h^{\F}]_{\tau}\,ds_{\vecx}\\
			&=-\int_{\Thx}e_h^{\elec}\cdot\rot\zeta_h^{\D}\,d\vecx+\int_{\Thx}e_h^{\magn}\cdot\rot\zeta_h^{\F}\,d\vecx\\
			&-\int_{\edgesx}\widetilde{e_h^{\elec}}\cdot[\zeta_h^{\D}]_{\tau}\,ds_{\vecx}+\int_{\edgesx}\widetilde{e_h^{\magn}}\cdot[\zeta_h^{\F}]_{\tau}\,ds_{\vecx}\\
			&+\int_{\Thx}(\rot\elec)\cdot\zeta_h^{\D}\,d\vecx-\int_{\Thx}(\rot\magn)\cdot\zeta_h^{\F}\,d\vecx.
		\end{align*}
By Lemma \ref{lem:approximation_lemma},
		$$
			\left|\int_{\Thx}(e_h^{\elec})\cdot\rot\zeta_h^{\D}\,d\vecx\right| \leq C h^k\norma{e_h^{\elec}}_{0,\Omega_x}\norma{\D}_{k+1,\Omega_x},$$
		$$
			\left|\int_{\Thx}e_h^{\magn}\cdot\rot\zeta_h^{\F}\,d\vecx\right|\leq C h^k\norma{e_h^{\magn}}_{0,\Omega_x}\norma{\F}_{k+1,\Omega_x},
$$
$$
			 \left|\int_{\edgesx}(\widetilde{e^{\elec}_h})\cdot[\zeta^{\D}_h]_{\tau}-(\widetilde{e^{\magn}_h})\cdot[\zeta^{\F}_h]_{\tau}\,ds_{\vecx} \right|\leq C h^{k+1/2}\left(\norma{\D}_{k+1,\Omega_x}+\norma{\F}_{k+1,\Omega_x}\right) \left(\norma{e^{\elec}_h}_{0,\edgesx}+\norma{e^{\magn}_h}_{0,\edgesx}\right).
			$$
Now notice that  
		\begin{align*}
			\norma{e^{\elec}_h}_{0,\edgesx}&\leq \norma{\epsilon^{\elec}_h}_{0,\edgesx}+\norma{\zeta^{\elec}_h}_{0,\edgesx}\\
			&\leq C[ h^{-1/2}\norma{\epsilon^{\elec}_h}_{0,\Omega_x}+h^{k+1/2}   ]\\
			&\leq Ch^{-1/2}[\norma{e^{\elec}_h}_{0,\Omega_x}+h^{k+1} ].
		\end{align*}
		Analogously 
		\begin{equation*}
			\norma{e^{\magn}_h}_{0,\edgesx} \leq C h^{-1/2}[\norma{e_h^{\magn}}_{0,\Omega_x}+h^{k+1} ].
		\end{equation*}
	Therefore,
		\begin{align*}
			\left|\int_{\edgesx}(\widetilde{e^{\elec}_h})\cdot[\zeta^{\D}_h]_{\tau}-(\widetilde{e^{\magn}_h})\cdot[\zeta^{\F}_h]_{\tau}\,ds_{\vecx} \right| 
			\leq Ch^{k}\left(\norma{\D}_{k+1,\Omega_x}+\norma{\F}_{k+1,\Omega_x}\right)\left(\norma{e_h^{\elec}}_{0,\Omega_x}+\norma{e_h^{\magn}}_{0,\Omega_x}+h^{k+1}  \right).
 		\end{align*}
		Now by the properties of the orthogonal projection $\mathbf{\Pi}_x^k$
		\begin{align*}
			\left|\int_{\Thx}(\rot\elec)\cdot\zeta_h^{\D}\,d\vecx\right|&=\left|\int_{\Thx}(\rot\elec-\mathbf{\Pi}_x^k(\rot\elec))\cdot\zeta_h^{\D}\,d\vecx\right| 
			\leq C  h^{2k+2}\norma{\D}_{k+1,\Omega_x},
		\end{align*}
		where $C$ depends on $\norma{\elec}_{k+2,\Omega_x}.$
By an analogous procedure
		\begin{equation*}
			\left|\int_{\Thx}(\rot\magn)\cdot\zeta_h^{\F}\,d\vecx\right|\leq C  h^{2k+2}\norma{\F}_{k+1,\Omega_x},
		\end{equation*}
where $C$ depends on		$\norma{\magn}_{k+2,\Omega_x}.$
Putting all the above calculations together, we arrive at, 
		\begin{align}
		\label{eq:residual_b_estimate}
			&|b_h(\elec_h,\magn_h;\zeta_h^F,\zeta_h^D)| 
			\leq Ch^{k}\left(\norma{\D}_{k+1,\Omega_x}+\norma{\F}_{k+1,\Omega_x}\right)\left(\norma{e_h^{\elec}}_{0,\Omega_x}+\norma{e_h^{\magn}}_{0,\Omega_x}+h^{k+1}\right),
		\end{align}
				where $C$ depends on $\norma{\elec}_{k+2,\Omega_x}, \norma{\magn}_{k+2,\Omega_x}.$

		We will deal now with the term $a_h$, which is
		\begin{equation}
			a_h(f_h,\elec_h,\magn_h,\zeta_h^{\varphi})=a_{h,1}(f_h,\zeta_h^{\varphi})+a_{h,2}(f_h,\elec_h,\magn_h;\zeta_h^{\varphi}).\label{eq:b_h_estimate}
		\end{equation}
	First, we have
		\begin{equation*}
			a_{h,1}(f_h;\zeta_h^{\varphi})=\int_{\Th}e_h^f\vecv\cdot\nabla_{\vecx}\zeta_h^{\varphi}\,d\vecx d\vecv+\int_{\Thv}\int_{\edgesx}\widetilde{e_h^f\vecv}[\einspace{\varphi}]_x\,ds_{\vecx}d\vecv -\int_{\Th}\nabla_{\vecx}f\cdot\vecv\einspace{\varphi}\,d\vecx d\vecv
		\end{equation*}
		The first term can be easily bounded, by using Lemma \ref{lem:approximation_lemma}. 
				\begin{align*}
			\left|\int_{\Th}e_h^f\vecv\cdot\nabla_{\vecx}\zeta_h^{\varphi}\,d\vecx d\vecv\right|& \leq C  h^{k}\norma{e_h^f}_{0,\Omega}\norma{\varphi}_{k+1,\Omega}.
		\end{align*}
		Similarly,
				\begin{align*}
			\left|\int_{\Thv}\int_{\edgesx}\widetilde{e_h^f\vecv}[\einspace{\varphi}]_x\,ds_{\vecx}d\vecv\right|& 			\leq C \norma{e_h^f}_{\Thv\times\edgesx}\norma{\einspace{\varphi}}_{\Thv\times\edgesx}\\
			&\leq C h^{k+1/2}\norma{e_h^f}_{\Thv\times\edgesx}\norma{\varphi}_{k+1,\Omega}\\
			&\leq Ch^{k+1/2}(\norma{\varepsilon_h^f}_{\Thv\times\edgesx}+\norma{\zeta_h^f}_{\Thv\times\edgesx})\norma{\varphi}_{k+1,\Omega}\\
			&\leq Ch^{k}(\norma{e_h^f}_{0,\Omega}+h^{k+1} )\norma{\varphi}_{k+1,\Omega}.
					\end{align*}

For the last term notice that by the properties of the projection $\Pi^k$ and the fact that $\Pi^k(\nabla_{\vecx}f\cdot\vecv)$ is a polynomial of degree $k$,  
		\begin{align*}
			\int_{\Th}\nabla_{\vecx}f\cdot\vecv\einspace{\varphi}\,d\vecx d\vecv&=\int_{\Th}(\nabla_{\vecx}f\cdot\vecv-\Pi^k(\nabla_{\vecx}f\cdot\vecv))\einspace{\varphi}\,d\vecx d\vecv\\
						&\leq C h^{2k+2}\norma{\varphi}_{k+1,\Omega},
		\end{align*}
		where $C$ depends on $\norma{f}_{k+2,\Omega}.$ 
		By using all the calculations above, we can conclude that 
		\begin{equation}
			|a_{h,1}(f_h;\einspace{\varphi})|\leq Ch^{k}\norma{e_h^f}_{0,\Omega}\norma{\varphi}_{k+1,\Omega} +Ch^{2k+1} \norma{\varphi}_{k+1,\Omega},
			\label{eq:residual_estimate_a_h_1}
		\end{equation}
				where $C$ depends on $\norma{f}_{k+2,\Omega}.$ 
		To conclude our proof, we only need to bound $a_{h,2}$, this time we will do things a little bit different, notice that
		\begin{equation*}
			a_{h,2}(f_h,\elec_h,\magn_h,\zeta_h^{\varphi})=a_{h,2}(f,\elec_h,\magn_h,\zeta_h^{\varphi})-a_{h,2}(e_h^f,\elec_h,\magn_h,\zeta_h^{\varphi}).
		\end{equation*}
		We will get started by noting that $\widetilde{f(\elec_h+\vecv\times\magn_h)}=f\{\elec_h+\vecv\times\magn_h\}_{v}=f\left(\elec_h+\vecv\times\magn_h\right)$, then
		\begin{align*}
			a_{h,2}(f,\elec_h,\magn_h,\zeta_h^{\varphi})&=-\int_{\Th}f(\elec_h+\vecv\times\magn_h)\cdot\nabla_{\vecv}\einspace{\varphi}\,d\vecx d\vecv+\int_{\Thx}\int_{\edgesv}f(\elec_h+\vecv\times\magn_h)\cdot[\einspace{\varphi}]_{v}\,d\vecx d\vecv\\
			&=\int_{\Th}f(e^{\elec}_h+\vecv\times e^{\magn}_h)\cdot\nabla_{\vecv}\einspace{\varphi}\,d\vecx d\vecv-\int_{\Thx}\int_{\edgesv}f(e^{\elec}_h+\vecv\times e^{\magn}_h)\cdot[\einspace{\varphi}]_{v}\,d\vecx d\vecv\\
			&+\int_{\Th}\nabla_{\vecv}f\cdot(\elec+\vecv\times\magn)\einspace{\varphi}\,d\vecx d\vecv.
		\end{align*}
		We obtained the last inequality by adding and subtracting $\int_{\Th}f(\elec+\vecv\times\magn)\cdot\nabla_{\vecv}\einspace{\varphi}\,d\vecx d\vecv$ , integration by parts, and the fact that $\nabla_{\vecv}\cdot(\elec+\vecv\times\magn)=0.$ in this way 
		\begin{align*}
			\left|\int_{\Th}f(e^{\elec}_h+\vecv\times e^{\magn}_h)\cdot\nabla_{\vecv}\einspace{\varphi}\,d\vecx d\vecv\right|& 			 \leq C  h^{k}(\norma{e^{\elec}_h}_{0,\Omega_x}+\norma{e^{\magn}_h}_{0,\Omega_x})\norma{\varphi}_{k+1,\Omega},
		\end{align*}
		and
		\begin{align*}
			\left|\int_{\Thx}\int_{\edgesv}f(e^{\elec}_h+\vecv\times e^{\magn}_h)\cdot[\einspace{\varphi}]_{v}\,d\vecv d\vecx \right|& \leq C h^{k+1/2}(\norma{e_{h}^\elec}_{0,\Omega_x}+\norma{e^{\magn}_h}_{0,\Omega_x})\norma{\varphi}_{k+1,\Omega}.
		\end{align*}
		 Last but not least by the same arguments as previous estimates
		\begin{align*}
			\int_{\Th}\nabla_{\vecv}f\cdot(\elec+\vecv\times\magn)\einspace{\varphi}\,d\vecx d\vecv&=\int_{\Th}(\nabla_{\vecv}f\cdot(\elec+\vecv\times\magn)-\Pi^k \nabla_{\vecv}f\cdot(\elec+\vecv\times\magn))\einspace{\varphi}\,d\vecx d\vecv\\
			&\le C h^{2k+2}\norma{\varphi}_{k+1,\Omega},
		\end{align*}
			where $C$ depends on $\norma{f}_{k+2,\Omega}, \norma{\elec}_{k+1,\Omega_x}, \norma{\magn}_{k+1,\Omega_x}.$
		We can conclude that 
		\begin{equation}
			|a_{h,2}(f,\elec_h,\magn_h;\einspace{\varphi})|\leq C h^{k}(\norma{e_{h}^\elec}_{0,\Omega_x}+\norma{e^{\magn}_h}_{0,\Omega_x})\norma{\varphi}_{k+1,\Omega}+Ch^{2k+2}\norma{\varphi}_{k+1,\Omega}.
			\label{eq:residual_a_h_2_estimate_p2}
		\end{equation}
		Finally we just need to estimate  
		\begin{equation*}
			a_{h,2}(e_h^f,\elec_h,\magn_h;\einspace{\varphi})=-\int_{\Th}e_h^f\left(\elec_h+\vecv\times\magn_h\right)\cdot\nabla_{\vecv}\einspace{\varphi}\,d\vecx d\vecv+\int_{\Thx}\int_{\edgesv}\widetilde{e_h^f(\elec_h+\vecv\times\magn_h)}\cdot[\einspace{\varphi}]_{v}\,ds_{v}d\vecx
		\end{equation*} 
We have
		\begin{align*}
			&\left|\int_{\Th}e_h^f\left(\elec_h+\vecv\times\magn_h\right)\cdot\nabla_{\vecv}\einspace{\varphi}\,d\vecx d\vecv\right|			\leq C  h^k\norma{e_h^f}_{0,\Omega}(\norma{\elec_h}_{0,\infty,\Omega_x}+\norma{\magn_h}_{0,\infty,\Omega_x})\norma{\varphi}_{k+1,\Omega}\\
			&\leq C  h^{k}\norma{e_h^f}_{0,\Omega}(\norma{\epsilon^{\elec}_h}_{0,\infty,\Omega_x}+\norma{\epsilon^{\magn}_h}_{0,\infty,\Omega_x}+\norma{\Pi^k_x\elec}_{0,\infty,\Omega_x}+\norma{\Pi^k_x\magn}_{0,\infty,\Omega_x})\norma{\varphi}_{k+1,\Omega}\\
			&\leq C  h^{k-d_x/2}\norma{e_h^f}_{0,\Omega}(\norma{\epsilon^{\elec}_h}_{0,\Omega_x}+\norma{\epsilon^{\magn}_h}_{0,\Omega_x})\norma{\varphi}_{k+1,\Omega}+C h^{k}\norma{e_h^f}_{0,\Omega}(\norma{\elec}_{0,\infty,\Omega_x}+\norma{\magn}_{0,\infty,\Omega_x})\norma{\varphi}_{k+1,\Omega}\\
			&\leq C  h^{k-d_x/2}\norma{e_h^f}_{0,\Omega}(\norma{e^{\elec}_h}_{0,\Omega_x}+\norma{e^{\magn}_h}_{0,\Omega_x}+h^{k+1})\norma{\varphi}_{k+1,\Omega}+C  h^{k}\norma{e_h^f}_{0,\Omega}\norma{\varphi}_{k+1,\Omega}\\
			&\leq C h^{k}\norma{e_h^f}_{0,\Omega}(h^{-d_x/2}\norma{e^{\elec}_h}_{0,\Omega_x}+h^{-d_x/2}\norma{e^{\magn}_h}_{0,\Omega_x}+1)\norma{\varphi}_{k+1,\Omega},
		\end{align*}
	    Here we used the fact that  whenever $d_x=1,\,2,\,3$, $k+1-d_x/2>0$, Lemma \ref{lem:inverse_inequalities} and the fact that 
	    $\Pi_x$ is bounded in any $L^p$-norm $(1\leq p\leq \infty)$ \cite{crouzeix1987stability,ayuso2009discontinuous},
	    \begin{equation*}
	    	\norma{\Pi_x\elec}_{0,\infty,\Omega_x}\leq C \norma{\elec}_{0,\infty,\Omega_x}, 	    	\norma{\Pi_x\magn}_{0,\infty,\Omega_x}\leq C \norma{\magn}_{0,\infty,\Omega_x}.
	    \end{equation*}
		Finally  		\begin{align*}
			&\int_{\Thx}\int_{\edgesv}\widetilde{e_h^f(\elec_h+\vecv\times\magn_h)}\cdot[\einspace{\varphi}]_{v}\,ds_{v}d\vecx\\
			&\leq C h^{k+1/2}(\norma{\elec_h}_{0,\infty,\Omega_x}+\norma{\magn_h}_{0,\infty,\Omega_x}) \norma{e_h^f}_{0,\Thx\times\edgesv}\norma{\varphi}_{k+1,\Omega}\\
			&\leq C h^{k+1/2}(\norma{\elec_h}_{0,\infty,\Omega_x}+\norma{\magn_h}_{0,\infty,\Omega_x}) h^{-1/2}(\norma{e_h^f}_{0,\Th}+h^{k+1}\norma{f}_{k+1,\Omega}) \norma{\varphi}_{k+1,\Omega}\\	
			&\leq C h^{k}(\norma{\elec_h}_{0,\infty,\Omega_x}+\norma{\magn_h}_{0,\infty,\Omega_x}) (\norma{e_h^f}_{0,\Th}+h^{k+1}\norma{f}_{k+1,\Omega}) \norma{\varphi}_{k+1,\Omega}\\	
			&\leq C h^{k}(\norma{e_h^f}_{0,\Th}+h^{k+1})(h^{-d_x/2}\norma{e^{\elec}_h}_{0,\Omega_x}+h^{-d_x/2}\norma{e^{\magn}_h}_{0,\Omega_x}+1)\norma{\varphi}_{k+1,\Omega}.
		\end{align*} 
		In this way we conclude that 
		\begin{align}
		\label{eq:residual_a_h_2_estimate_p2c}
			 |a_{h,2}(e_h,\elec_h,\magn_h;\einspace{\varphi})|
			  \leq C h^{k}(\norma{e_h^f}_{0,\Th}+h^{k+1})(h^{-d_x/2}\norma{e^{\elec}_h}_{0,\Omega_x}+h^{-d_x/2}\norma{e^{\magn}_h}_{0,\Omega_x}+1)\norma{\varphi}_{k+1,\Omega}	\end{align}
		Then by putting together \eqref{eq:residual_b_estimate}, \eqref{eq:residual_estimate_a_h_1}, \eqref{eq:residual_a_h_2_estimate_p2}, \eqref{eq:residual_a_h_2_estimate_p2c}, and using Theorem \ref{thm:main_approx_result}, we have
		\begin{align*}
		&|a_h(f_h,\elec_h,\magn_h;\zeta_h^{\varphi})+b_h(\elec_h,\magn_h;\zeta_h^{\F},\zeta_h^{\D})|\\
		 & \leq  Ch^{k}\left(\norma{\D}_{k+1,\Omega_x}+\norma{\F}_{k+1,\Omega_x}+\norma{\varphi}_{k+1,\Omega}\right)\left(\norma{e_h^{\elec}}_{0,\Omega_x}+\norma{e_h^{\magn}}_{0,\Omega_x}+h^{k+1}  \right)	\\
		 &+ C h^{k}(\norma{e_h^f}_{0,\Th}+h^{k+1})(h^{-d_x/2}\norma{e^{\elec}_h}_{0,\Omega_x}+h^{-d_x/2}\norma{e^{\magn}_h}_{0,\Omega_x}+1)\norma{\varphi}_{k+1,\Omega}\\
		 & \leq  Ch^{2k+1/2}\left(\norma{\D}_{k+1,\Omega_x}+\norma{\F}_{k+1,\Omega_x}+\norma{\varphi}_{k+1,\Omega}\right).	
		\end{align*}
	where we have used $k+1/2-d_x/2>0.$	
	An application of Cauchy-Schwarz inequality concludes the proof.
		
\section{Proof of Lemma \ref{lem:consistency} }
	The terms inside the integral of $\Theta_{\rm C}$ can be split in $I+II$, where
		\begin{align*}
			&I=(f_h,\varphi_t)_{\Omega}-a_h(f_h,\elec_h,\magn_h;\varphi)+l_h(\curre_h,\F)\\
			&II=(\elec_h,\F_t)_{\Omega_x}+(\magn_h,\D_t)_{\Omega_x}-b_h(\elec_h,\magn_h;\F,\D)\mathcal+\mathcal{F}(f,\elec,\magn;\varphi)
		\end{align*}
		since $\varphi$ is a smooth function, $[\varphi]_{x}=0$ and $[\varphi]_v=0$, in this way, by using \eqref{eq:dual_distribution}, and the definition of $l_h$, we conclude that, 
		\begin{align*}
			I&=(f_h,-\vecv\cdot\nabla_{\mathbf{x}}\varphi-(\elec+\vecv\times\magn)\cdot\nabla_{\vecv}\varphi+\vecv\cdot\F)_{\Omega}-a_h(f_h,\elec_h,\magn_h;\varphi)+l_h(\curre_h;\F)\\
			&=-\int_{\Th}f_h\vecv\cdot\nabla_{\vecx}\varphi\,d\vecx d\vecv-\int_{\Omega}f_h(\elec+\vecv\times\magn)\cdot\nabla_{\vecv}\varphi\,d\vecx d\vecv-l_h(\curre_h;\F)\\
			&+\int_{\Th}f_h\vecv\cdot\nabla_{\vecx}\varphi\,d\vecx+\int_{\Omega}f_h(\elec_h+\vecv\times\magn_h)\cdot\nabla_{\vecv}\varphi\,d\vecx d\vecv+l_h(\curre_h;\F)\\
			&=-\int_{\Omega}f_h(e_h^{\elec}+\vecv\times e_h^{\magn})\cdot\nabla_{\vecv}\varphi\,d\vecx\,d\vecv.
		\end{align*}
		On the other hand, by using \eqref{eq:dual_elec} and \eqref{eq:dual_magn}, since $\F$ and $\D$ are smooth functions $[\F]_{\tau}=[\D]_{\tau}=0$,we have that
		\begin{align*}
			II&=(\elec_h,\rot \D)_{\Thx}-(\magn_h,\rot \F)_{\Thx}-b_h(\elec_h,\magn_h;\F,\D)+\mathcal{F}(f,\elec,\magn;\varphi)\\
			  &-\int_{\Omega}f\elec_h\cdot\nabla_v\varphi\,d\vecx d\vecv+\int_{\Omega}f\magn_h\cdot(\vecv\times\nabla_v\varphi)\,d\vecx d\vecv\\
			&=(\elec_h,\rot \D)_{\Thx}-(\magn_h,\rot \F)_{\Thx}-(\elec_h,\rot \D)_{\Thx}+(\magn_h,\rot \F)_{\Thx}\\
			&-\int_{\Omega}f(\elec_h+\vecv\times\magn_h)\cdot\nabla_{\vecv}\varphi\,d\vecx d\vecv+\int_{\Omega}f(\elec+\vecv\times\magn)\cdot\nabla_{\vecv}\varphi\,d\vecx d\vecv\\
			&=\int_{\Omega}f(e_h^{\elec}+\vecv\times e_h^{\magn})\cdot\nabla_{\vecv}\varphi\,d\vecx d\vecv.
		\end{align*}
		We obtain
	\begin{align*}
			I+II&=\int_{\Omega}e_h^f(e_h^{\elec}+\vecv\times e_h^{\magn})\cdot\nabla_{\vecv}\varphi\,d\vecx d\vecv\\
		 &\leq C \norma{e_h^f}_{\Omega}(\norma{e_h^{\elec}}_{\Omega_x}+\norma{e_h^{\magn}}_{\Omega_x})\norma{\nabla_{\mathbf{v}}\varphi}_{\infty,\Omega}\\
		&\leq C \norma{e_h^f}_{\Omega}(\norma{e_h^{\elec}}_{\Omega_x}+\norma{e_h^{\magn}}_{\Omega_x})\norma{\varphi}_{k+1,\Omega}
	\end{align*}
where we used the Sobolev inequality \cite{brenner2008mathematical}, $\norma{\nabla_{\mathbf{v}}\varphi}_{\infty,\Omega}\leq C \norma{\varphi}_{k+1,\Omega},$ which requires $k>(d_x+d_v)/2$. Using Theorem \ref{thm:main_approx_result}, we conclude the proof.
\end{appendices}

	\bibliographystyle{abbrv}
	\bibliography{XBib,refer,ref_cheng_plasma_2,ref_cheng,ref_cheng_2}
\end{document}